\documentclass[11pt]{amsart}
\usepackage{amscd}
\usepackage{amsfonts}
\usepackage{amsmath,amsthm,hyperref}
\usepackage[all]{xy}
\usepackage{amsmath,amsthm,hyperref}
\usepackage{amsmath,amssymb,amsthm,latexsym}
\usepackage{amscd}
\usepackage{xcolor}
\usepackage{multicol}
\usepackage{graphicx}
\usepackage{bm}
\usepackage{setspace}
\usepackage{caption}
\newtheorem{theorem}{Theorem}[section]
\newtheorem{proposition}[theorem]{Proposition}
\newtheorem{definition}[theorem]{Definition}
\newtheorem{corollary}[theorem]{Corollary}
\newtheorem{lemma}[theorem]{Lemma}

\numberwithin{equation}{section}
\theoremstyle{remark}
\newtheorem{remark}[theorem]{Remark}

\newcommand{\R}{\mathbb{R}}
\newcommand{\C}{\mathbb{C}}
\newcommand{\D}{\mathbb{D}}

\newcommand{\T}{\mathbb{T}}
\newcommand{\Up}{\mathbb{H}}
\newcommand{\St}{\mathbb{S}}
\newcommand{\Z}{\mathbb{Z}}
 \newcommand{\dd}{\mathrm{d}}
\newcommand{\M}{\mathbb{M}}

\begin{document}
\title[Equivariant  primitive harmonic maps into  $k-$symmetric spaces]{Equivariant  primitive harmonic maps into  $k-$symmetric spaces, with applications to Willmore surfaces}
\author{Josef F. Dorfmeister, Peng Wang }
%\date{\today}
\maketitle

\begin{abstract}  In this note we discuss the construction of equivariant primitive harmonic maps into  $k-$symmetric spaces and give many applications to the construction of Willmore surfaces. In particular,  examples of $S^1-$equivariant  Willmore Moebius strips in $S^3$ are obtained as generalizations of the  Lawson minimal Klein Bottles in $S^3$.
\end{abstract}

\vspace{0.5mm}  {\bf \ \ ~~Keywords: equivariant primitive harmonic maps;  Willmore sufaces; Willmore Moebius strips  }\vspace{2mm}

{\bf\   ~~ MSC(2020): \hspace{2mm} 53C40, 53A31}

%\tableofcontents

\section{Introduction}

In every class of surfaces, those surfaces with  symmetries are of particular interest and importance (see for instance \cite{Bu-Ki,Do-Ha2,Do-Ha4,Ejiri-equ,FP}).  In this paper we concentrate on surfaces containing a one--parameter group of (necessarily orientation preserving)  symmetries. More precisely, we investigate exclusively surfaces which are characterized by the harmonicity of some ``Gauss map" from a Riemann surface into a real symmetric space admitting an equivariant action of a one-parameter  group of symmetries. This is motivated mainly by the study of Willmore surfaces with symmetries by the authors \cite{DoWaSym1,DoWaSym2} and the idea of generalizing the beautiful theory of Burstall-Killian on $S^1-$equivariant harmonic tori in $S^2$ \cite{Bu-Ki} for constant mean curvature surfaces.

On the other hand, in the celebrated paper \cite{Lawson} Lawson constructs infinitely many minimal Klein bottles in $S^3$. These examples are all $S^1-$equivariant. It is a natural idea to understand these examples from the point view of Willmore surfaces. Moreover, with such an expression of Lawson's classical examples at hand, one would be able to consider the generalization of Lawson's classical examples by using the loop group method \cite{DoWa10,DoWa12,DoWaSym1,DoWaSym2}. This forms the second motivation of this paper. And we do obtain infinitely many $S^1-$Willmore M\"{o}bius strips in $S^3$, which are not minimal in any space form and are new to the authors' best knowledge.

We start by giving a very general definition for an equivariant map from  a Riemann surface to a $k-$symmetric space (see Definition \ref{def-equi}.) It turns out that  up to biholomorphic transformations, there are only two types of such maps, namely those equivariant under a one-parameter family of translations and those  equivariant under a one-parameter family of rotations.

 Accordingly we split up the discussion into the case of ``translationally equivariant'' harmonic maps, the ``TE case'',  and of  ``rotationally equivariant '' harmonic maps, the ``RE case''.

Then we focus on TE harmonic maps in this paper. To be concrete, we recall the basic loop group theory for TE primitive harmonic maps in a $k-$symmetric space developed by Burstall and Kilian \cite{Bu-Ki}.  The key idea is to assign a nice constant holomorphic potential to a TE primitive harmonic map (See Section 4 for more details). Then we investigate properties of a TE primitive harmonic map via this potential. In particular,  we use it to discuss  periodicity properties of TE harmonic maps.

For  applications we concentrate primarily on equivariant Willmore surfaces in this paper, last not least, since for many other surface classes the equivariant maps have already been studied in detail (See e.g., \cite{Bu-Ki,Do-Ha2,Do-Ha3,Do-Ha4}). Note that in \cite{FP} a detailed discussion of equivariant Willmore surfaces in $S^3$ has been  presented, where the authors  mainly use the work of Hsiang and Lawson on minimal submanifolds of low co-homogeneity \cite{HL}. We also refer to \cite {heller2014} for a recent treatment of equivariant constrained Willmore tori in $S^3$.

For our discussions two results are crucial. First we show that our definition of a "translationally equivariant surface $f$" implies the existence of a frame $F$ satisfying the quasi-invariant transformation formula $F(z + t) = R(t) F(z)$. Secondly, starting from such a formula one can apply almost verbatim the work of \cite{Bu-Ki} and obtain a description of these surfaces which is closely analogous to the discussion of Delaunay surfaces in the case of constant mean curvature surfaces in $R^3$.

In our explicit applications of the basic theory of TE primitive harmonic maps
into k-symmetric spaces we consider Willmore cylinders and Willmore Moebius strips in $S^m$.
For both types of Willmore surfaces we present characterizing properties for the corresponding Delaunay type matrices. And for both types of surfaces we give new examples.

 We plan to discuss Willmore tori in more detail in a subsequent publication.

This paper is organized as follows: In Section 2 we discuss some basic concepts in the
investigation of equivariant maps from Riemann surfaces to homogeneous spaces $G/K$. Then we consider translationally equivariant  (TE for short) maps in Section 3. The loop group theory of TE harmonic maps into $k-$symmetric spaces is introduced in Section 4 and is applied to the study of Willmore surfaces in Section 5.  Then in Section 6 we consider  periodicity properties of  TE harmonic maps. In particular, we discuss how to construct equivariant cylinders with special emphasis on Willmore cylinders. Section 7 is an application of the above theory to Willmore Moebius surfaces.
 We give a characterization of all TE Willmore Moebius surfaces in $S^m$
 in terms of Delaunay type potentials, see Theorem \ref{thm-equ-n-ori} and
 Theorem \ref{thm-equ-n-ori-2}
  In particular we obtain infinitely many new examples of equivariant  Willmore Moebius strips in $S^3$, which are generalizations of the  Lawson minimal Klein Bottles in $S^3$. These new surfaces are {\em not}  conformally equivalent to any minimal surface in any space form.

%%%%%%%%%%%%%%%%%%%%%%%%%
\section{Equivariant maps from Riemann surfaces to $k-$symmetric spaces}

We start with a very general definition and recall from Riemann surface theory all possibilities.

%%%%%%%%%%%%%%%%%
\subsection{Basic definitions}

\emph{In this paper we will consider differentiable maps $f: M \rightarrow G/K$, where $M$ is a connected Riemann surface, $G$ is a real Lie group and $K$ is a closed subgroup of $G$ which contains the center of $G$. We will always assume that $G$ is represented as a matrix group. We will  {specialize} the setting further starting in Section 4.}

We begin  by stating a simple lemma.

\begin{lemma}
Let $G$ be a real Lie group and $K$ a closed subgroup as above. Let $M$ denote a
connected Riemann surface and $f: M \rightarrow G/K$  a differentiable map.
Let $\mathcal{C}(f)$ denote the closed subgroup \[\mathcal{C}(f) = \{ R \in G| R.f(p) = f(p) \hspace{2mm} \mbox{ for all } \hspace{2mm}p \in M \}\] of $G$. Then we observe
\begin{enumerate}
\item If for some $g \in Aut(M)$ there exists  some $R \in G$ such
that $f(g.p) = R.f(p)$ for all $p \in M$, then $R^{-1}\mathcal{C}(f)R \subset \mathcal{C}(f)$
and exactly the  elements
  $ R' \in R\mathcal{C}(f)$  satisfy $f(g.p) = R'.f(p)$.

 \item If $R \in G$ is given, and $g \in Aut(M)$ satisfies $f(g.p) = R.f(p)$ for all $p \in M$, then also any $g' \in  \kappa(f) g$ satisfies $f(g'.p) = R.f(p)$ ,
 where $\hbox{$\kappa(f) = \{h \in Aut(M)| f(hp) = f(p)$ for all $p \in M$\}.}$

\item  The center of $G$ is contained in $\mathcal{C}(f)$.
 \end{enumerate}
 \end{lemma}
\begin{proof}
(1). Since $R^{-1} \tilde{R}R.f(p)=R^{-1}\tilde{R}. f(g.p)=R^{-1}.f(g.p)=f(p)$ for any $\tilde{R}\in \mathcal{C}(f)$, we obtain $R^{-1}\mathcal{C}(f)R \subset \mathcal{C}(f)$. And the last statement holds due to $f(g.p) = R'.f(p)$ being equivalent to $f(p)=(R^{-1}R').f(p)$.

(2). Assume $g'=hg$ with $h\in\kappa(f)$. We  have $f(g'.p)=f(hg.p)=f(g.p)=R.f(p)$.

(3) Obvious.
\end{proof}

\begin{definition} \label{full}
A differentiable map $f: M \rightarrow G/K$ is called full, if  $\mathcal{C}(f)$ coincides with the center of $G$.
 More precisely this means
$\mathcal{C}(f) = \{ R \in G; R.f(p) = f(p) \} =$ center of G.
\end{definition}

\begin{remark}
\emph{In this paper we will always  consider  full maps, unless the opposite is stated explicitly. } This will yield  particularly simple results if the center of $G$ only consists of the identity element $e \in G$.
In our applications to Willmore surfaces we will use $SO^+(1,n+3),$ the connected component of $SO(1,n+3)$ containing the identity element, and here the center is trivial.
\end{remark}

\begin{definition}\label{def-equi}
Let $G$ be a real Lie group and $K$ a closed subgroup. Let $M$ denote a
connected Riemann surface and $f: M \rightarrow G/K$  a differentiable map.
Let  $g_t$ be a one--parameter group of biholomorphic maps  of $M$. Then $f$ is called {\bf equivariant }with regard to $g_t$ if there exists a one--parameter group $R(t):\R\rightarrow G$ such that
\begin{equation} \label{equiv-gen}
f(g_t.p) = R(t) f(p)
\end{equation}
for all $ p \in M$ and $t \in \R$.
\end{definition}

The following lemma indicates the reason for why we  restrict to full maps:
\begin{lemma}  If $f: M \rightarrow G/K$ is equivariant relative to
 $(g_t,R(t))$, then $f,$ not necessarily assumed to be full, is also equivariant relative to $(g_t, \hat{R}(t))$ if $\hat{R} = S R(t) S^{-1}$  and
$S \in \mathcal{C}(f)$.
\end{lemma}

%%%%%%%%%%%%%%%%%%
\subsection{Riemann surfaces with  one-parameter groups of automorphisms}

We recall from the theory of Riemann surfaces (see e.g. \cite{Farkas-Kra}, Theorem V.4.1):
\begin{theorem} \label{Kra}
The only Riemann surfaces $M$ admitting a non--discrete group of automorphisms are the surfaces biholomorphically equivalent to
\begin{enumerate}
\item the simply--connected Riemann surfaces $\C,  \D, \C P^1 \cong S^2$.

\item the surfaces with fundamental group $\mathbb{Z}$, $\C^{\star} = \C\setminus \lbrace 0 \rbrace, \mathbb D^{\star} =  \mathbb D\setminus \lbrace 0 \rbrace$, and
$ \mathbb D_r = \{ z \in \C, 0 < r < |z| < 1$, where  $0< r < 1$.\

\item the surfaces with fundamental group isomorphic to  $\mathbb{Z} \times \mathbb{Z}$,
 i.e. the tori $\{\T^2 = \C/ \mathcal{L}\}$, where $\mathcal{L}$ is
a rank $= 2$ lattice in $\C$.
\end{enumerate}
\end{theorem}

\begin{remark}
The motivation of this work  {arose out of investigations on} Willmore surfaces. Let's  {consider now} Willmore surfaces.
Since the projection $\pi: \tilde{M} \rightarrow M$, $\tilde{M}$ denoting the universal cover of $M$, is a covering map, by the monodromy theorem one can always lift the equivariant Willmore surface $f: M \rightarrow G/K$ to an equivariant Willmore surface  $\tilde{f} : \tilde{M} \rightarrow G/K$.
The theorem above and the statement just above, taken together,  imply that we need to consider in the first place equivariant Willmore surfaces defined on simply-connected Riemann surfaces, and then Willmore surfaces defined on (topological) cylinders and tori. We note that the last cases are quotients of $\C$ in the case $M \cong  \C^*$ and
$ M \cong \T^2$, and of $ \mathbb D$ in the cases $M \cong  \mathbb D^*, \mathbb{ D}_r$.\end{remark}

Next we collect what possibilities we have for $g_t$ on a simply connected Riemann surface $\tilde{M}$
(see e.g. also \cite{Smyth}).
\vspace{3mm}

{\bf {Case  1 : $\tilde{M}=\D = \Up$ is the upper half plane:}}
\vspace{3mm}

In this case  { $Aut(\Up) \cong P SL(2, \R)$ and we  can w.l.g.}  consider $g_t$ as a one--parameter group  in $SL(2,\R)$. Then $g_t = \exp (tA)$ with some matrix $A \in \mathfrak{sl}(2,\R)$. If the eigenvalues of A are real and different,  then  the Moebius transformation induced by $g_t$ is, up to conjugation, of the form \[\hbox{$g_t.z = e^{2at} z $ for some $a \in \R$.}\]  Then $\ln ( \Up) -\pi/2$ is the open strip in $\C$ between the lines $y=-\pi/2$ and $y= \pi/2$ and $g_t $ acts there by translations parallel to the real axis.

 If $A$ is nilpotent (i.e. both eigenvalues  of $a$ are real, but equal, whence $= 0$), then $g_t$ already acts by translation in $\Up$.

If the eigenvalues of $a$ are complex, then changing biholomorphically to the
open unit disk $\D$ we see that $g_t$ acts by rotation about $z =0$.
\vspace{2mm}

 {\bf{Case  {2 : }$\tilde M = \C$} is the complex plane:}
  \vspace{3mm}

We know that $g_t$ is of the form $g_t .z = a_t .z + b_t$.

If $b_t =0$ for all $t$, then we have a ``complex  rotation''.
If  $a_t = 1$ for all $t$, then $g_t$ acts by translations anyway and we can assume w.l.g. that this translation is parallel to the $x-$axis. It thus remains the case where neither $b_t$ nor $a_t$ are trivial. Writing $g_t = exp(t A)$ it is easy to verify that after conjugation by some $t-$independent element in $Aut(\C)$ the one--parameter group $g_t$ turns into a one--parameter group of ``complex rotations''.
\vspace{3mm}

{\bf{Case   3 : $\tilde M = S^2$  is the Riemann sphere:}}
\vspace{3mm}

In this case we consider the group $SL(2,\C)$ of biholomorphic transformations of $S^2$
and write $g_t = \exp (t A)$ with $A$ some matrix in $\mathfrak{sl}(2,\C)$.

If $A$ has two equal eigenvalues, then $A $ is nilpotent and w.l.g. we can assume that $A$ is upper triangular with real non-zero off-diagonal entry. But then after stereographic projection $g_t$ acts by translations parallel to the real axis.
In the remaining case, $A$ has two different eigenvalues. In this case $g_t$ acts, after some stereographic projection, on $\C$ by $z  \rightarrow \exp(2ta)z$, i.e. by a complex rotation.\\

From the discussion above one obtains Theorem \ref{class-equi} (including the groups of translations).

%%%%%%%%%%%%%%%%%%%%%%%%%

\begin{theorem} \label{class-equi}
(Classification of Riemann surfaces admitting 1-parameter groups of automorphisms with w.l.g. representative groups, e.g. \cite{Farkas-Kra})

\begin{enumerate}
\item
 $S^2$ , group of rotations about the $3-$axis;

\item $\C^*$, group of all maps $z \rightarrow e^{tw}z$ , for some fixed $w \in \C \setminus\{0\}$;

\item $\mathbb{D}^*, \mathbb{D}_r,$ group of all rotations about $0$;

\item    $\T^2=\C/ \mathcal{L}_{\tau}$, group of translations parallel to the $x-$axis;

 \item
  \begin{enumerate}
  \item $\C$, group of all real translations;

\item$\C$, group of all rotations about the origin $0$;

\item $\C$, group of all maps $z \rightarrow e^{tw}z$ ,  for some fixed $w \in \C \setminus\{0\}$;
\end{enumerate}
\item
  \begin{enumerate}\item $\mathbb{D}$, group of all rotations about the origin $0$;

\item  $\mathbb{D} \cong \mathbb{H}$, group of all real translations;

\item  $ \mathbb{D} \cong \mathbb{H} \cong \log  \mathbb{H} =\St$, the strip between $y=0$ and $y = \pi$,
group of all real translations.
\end{enumerate}
\end{enumerate}
Here $\mathcal{L}_{\tau}$ is the free group generated by the two translations $z\mapsto z+1$, $z\mapsto z+\tau$, $\mathrm{Im}\tau >0$ and is possibly rotated against the $x-axis$.
\end{theorem}

{Note that each complex cylinder is biholomorphically equivalent to $\C^*, \D^*$, or $\D_r, $
for some $0 < r < 1$.} We refer to  \cite{Farkas-Kra}, Section V, for more details.

It is easy to observe that up to biholomorphic equivalence the following five types of one-parameter groups acting on a Riemann surface $M$ occur:
\begin{enumerate}
\item all translations  on the strip
\begin{equation}\label{eq-strip}
  M=\St:=\{z\in\C| -\pi<Im z<\pi\}
\end{equation}
containing the real-axis;
\item all rotations  about $z=0$  and $0 \notin M$   with $M = \C^*, \D^*,\D_r;$
\item all rotations about $z=0$ and $0 \in M$  with  $M = \C, \D, S^2;$
\item all maps $z \rightarrow e^{tw}z$ , for some fixed $w$ in $\C\backslash\R$;
here $z \in \C$ or $z \in \C^*$;
\item  {all maps of the cosets $z {L}_{\tau} \rightarrow (z + at) {L}_{\tau},$
where $a$ is a fixed complex number and  $t \in \R$ arbitrary;} {Here $L_\tau$ is the lattice in $\C$ spanned by $1$ and  $\tau$.}
\end{enumerate}

In Case (2) the Riemann surface $M$ has fundamental group $\mathbb Z$ and the group action can be lifted to the
universal cover $\tilde{M}$, where it acts  as in Case (1).
Thus in these cases all equivariant maps can be obtained from Case (1), where the lifted map $\tilde{f}$ is periodic with real period.

In Case (3) we have rotations with fixed point (in $M$). It turns out that it is best to treat these cases directly. But one could also remove the fixed point $(z=0)$, arriving at Case (2). One needs to make sure though in this case that the solutions obtained on $M\setminus\{0\}$ extend smoothly to $z=0$.

In Case (4) one should remove the fixed point ($z=0$, if $M = \C$) and lift the group action to the universal cover $\tilde{M}$.
As far as we only look at a  {one-parameter} group action, this case is essentially the same for all $w \in \C^*$.

In Case (5) the action can be lifted to the universal cover $\tilde{M} = \C$ and acts there w.l.g. as in Case (1).

We finish this introductory section with a definition which covers  all equivariant
  cases depending on whether the equivariant group action has a fixed point or not.
\begin{definition}
Let $f: M \rightarrow G/K$  a differentiable map.
\begin{enumerate}
\item Let $M$ be a strip  $\St$ parallel to the real axis. The map  $f$ will be called {\bf ``translationally equivariant" (TE for short)}, if it is equivariant under all real translations.

\item Let $M$ be a Riemann surface containing the point $z=0$ and admitting all rotations about $z=0$ as (biholomorphic) automorphisms. The map $f$ will be called {\bf ``rotationally equivariant'' (RE for short)}, if it is equivariant under all  rotations about $z=0$.
    \end{enumerate}
\end{definition}

\begin{remark} Let $y:M\rightarrow S^n$ be a conformal immersion of a Riemann surface. If it has a one-parameter group symmetry $(\gamma(t), A(t))$ with $\gamma(t):M\rightarrow M$ and $A(t)$ taking values in $SO^+(1,n+1)$. Then $\gamma(t)$ takes values in $Aut(M)$. As a consequence, the universal covering of $M$ is one of the Riemann surfaces stated in Theorem \ref{class-equi} and $\gamma(t)$ is one of the corresponding one-parameter {groups.} In particular, there exists a local complex coordinate $z$ such that $Re(z)=t$ on $M$. As a corollary,  equivariant minimal surfaces in \cite{HL} and equivariant Willmore surfaces in \cite{FP}  are all included in our definitions here.

\end{remark}

\begin{remark} Rotationally equivariant Willmore surfaces appear naturally in the study of equivariant Willmore $2$-spheres and equivariant Willmore $\mathbb RP^2$, which we will consider in another publication.
\end{remark}
%%%%%%%%%%%%%%%%%%%%%%%%%

\section{TE maps}

In this section, unless the contrary is stated explicitly, we will always consider TE maps defined on the open strip
$\St$ \eqref{eq-strip} containing the real axis.

\subsection{Frames for TE maps}

This section is to show that  for any  TE map $f:\St \rightarrow G/K$ of Definition \ref{def-equi} there exists
a lift $F :\St \rightarrow G$ which is quasi-invariant (see definition below).

First we will lift the map $f$ to a map $F$ into $G$.
\begin{lemma}  If $f: \St \rightarrow G/K$  is  a differentiable map, then
there exists a map $F: \St \rightarrow G$ such that the following diagram is commutative.
\begin{equation*}  \xymatrix{
&\ G \ar[d]^{\pi} \\
                      \mathbb{S}\ \ar[r]_{\ \ f \ }\ar[ur]^{F}  &\ \  G/K\ .
}\end{equation*}
\end{lemma}
\begin{proof}
Consider the pullback $\hat\pi:\hat{G}=f^*G\rightarrow\St$ of the bundle $\pi:G{\rightarrow}G/K$ by $f$.
Since $\St$ is contractible, there is a section $\sigma$ of the bundle $\hat{G}\rightarrow\St$
such that the following diagram commutes
\begin{equation*}  \xymatrix{\hat{G}\ar[d]^{\hat\pi} \ar[r]_{\ \ \hat{F} \ }  &\ G \ar[d]^{\pi} \\
                      \mathbb{S}\ \ar[r]_{\ \ f \ ~~} &\ \  G/K\ .
}\end{equation*}
So there exists a lift $F=\hat{F}\circ\sigma:\St\rightarrow G$ of $f$.
\end{proof}
\begin{corollary} Assume $f$ is TE, i.e. $f$ satisfies
\begin{equation}
f(p+t) = R(t) f(p)
\end{equation}
then $F$ satisfies the relation
\begin{equation} \label{equi-F-gen}
F(p + t)=R(t)F(p)h(t,p)
\end{equation}
with $h(t,p)\in K$ for all $z\in \St$ and $t \in \R$.
\end{corollary}

\begin{definition}
The map $F$ just defined will be called {\it a frame for $f$}.
\end{definition}

In (\ref{equi-F-gen}) the quantity $h$ varies with the choice of $F$. For many purposes it is {convenient} not having to deal with the presence  of $h$ in equation (\ref{equi-F-gen}). Fortunately, we can remove $h$ for all TE maps. First we note that by a straightforward computation it is easy to verify that the property of
$R(t)$ being a  one--parameter group translates for $h$ into the ``cocycle condition":
\begin{equation} \label{crossedhom}
h(t+s,z)=h(s,z)  h(t,z+s),\ \hbox{ $z \in \St$ and $s,t \in \R$.}
\end{equation}
Here we write $h,$ as usual,  ``as a function of $z$'', but $h$ actually is, of course, a differentiable function  of $z$ and $\bar z$ and we interpret our way of writing this as
``$f$ is function of $x$ and $y$ ''or as ``$f$ is a function of $z$ and $\bar{z}$''. We claim that $h$ is a ``coboundary":

\begin{theorem} \label{cocycle-plus}
Retaining the assumptions and the notation introduced above, there exists some function $\rho: \St \rightarrow K$ such that for all
$z=x+iy \in \St$ and all $t \in \R$ we have
\begin{equation}
h(t,z)=\rho(z)\rho(z+t)^{-1}, \hbox{ with } \rho(z)=h(x,iy)^{-1}.
\end{equation}
\end{theorem}
\begin{proof} Reducing $z$ to $iy$ we have
\[h(t,iy)h(s,t+iy)=h(t+s,iy)\]
whence $h(s,t+iy)=h(t,iy)^{-1}h(t+s,iy).$
Put
\[\rho(z)=h(x,iy)^{-1}.\]
 Then
$\rho(z)\rho(z+t)^{-1}=h(x,iy)^{-1} h(x+t,iy)=h(t,x+iy). $
\end{proof}

\begin{remark}
Note that we can replace $\rho (z) $ above by $\rho (z) \rho (0)^{-1}$ and thus we can assume w.l.g. $\rho (0) = I $, if this is desired.
\end{remark}

From this we derive the important

\begin{theorem} \label{prop-lift}
For every TE map $f$ defined on a strip $\St$ we can assume w.l.g. that  $f$ has a  frame $F$
which satisfies the relation
\begin{equation} \label{equi-F}
F(z + t) = R(t) F(z)
\end{equation}
for all $z \in \St$ and $t \in \R$.
\end{theorem}

\begin{proof}
Choosing $\rho$ as in the theorem just above and replacing $F$ by $\hat{F}=F\rho$ we obtain
$
\hat{F}(z + t) = R(t) \hat{F}(z)
$
for all $z \in \St$ and $t \in \R$. Writing again $F$ for $\hat{F}$ we obtain the claim.
\end{proof}

\begin{remark}\
\begin{enumerate}
\item  In \cite{Bu-Ki} ``equivariant primitive harmonic maps" have been considered.
The definition of the notion of ``equivariant" used in \cite{Bu-Ki} requires for the frame the equation \eqref{equi-F}.  In our definition we start from a natural geometric notion of equivariance on the level
of "immersions". Then, by Theorem \ref{prop-lift}, it follows that one can assume w.l.g. that the equation
\eqref{equi-F} assumed in  \cite{Bu-Ki} actually holds.

\item Frames as above will be called {\it quasi-invariant}.  Note that by (\ref{equi-F}) we can write
 \begin{equation}\label{equi-F-DL}
    F(x,y) = \exp(xD)\cdot L(y)
 \end{equation}
with $R(x) = exp(xD)$  for some $D\in\mathfrak{g}=Lie(G)$ and $L(y) = F(0,y)$.
\end{enumerate}
\end{remark}

%%%%%%%%%%%%%%%%%

\subsection{Equivariant maps and Maurer-Cartan forms of quasi-invariant  frames}

Since we are interested in a detailed study of TE maps, we are interested in as much information as possible about the frame $F$ and/or its
Maurer-Cartan form $\alpha = F^{-1} \dd F$. The following theorem generalizes a result of Burstall and Kilian  (\cite{Bu-Ki}) \label{BuKi-y}.
\begin{theorem}
Let $\St$ be a strip as above and let $F$ denote the frame of some
differentiable  map $f: \St \rightarrow G/K$.
Then $f$ is TE if and only if the Maurer-Cartan form $ F^{-1} \dd F $of $F$ only depends on $y$, where we write  $ z = x + i y$ for points in $\St$.
\end{theorem}

\begin{proof}
We  follow closely  the proof given in \cite{Bu-Ki}.

\vspace{3mm}
{\bf{Map to MC form:}}\hspace{3mm} We know
 $ F(x,y) = \exp(xD)\cdot L(y)$
  for the frame of the equivariant map $f$ from $\St$ into $G/K$, with $\exp(x D)=R(x)$ and $L(y)=F(0,y)$.
Then
\begin{equation}\label{MC}
\alpha = F^{-1} \dd F = L(y)^{-1} D L(y) \dd x + L(y)^{-1} \partial_y L(y) \dd y
\end{equation}
and $\alpha$ is integrable and its coefficients only depend on $y$.\vspace{3mm}

{\bf{MC form to  map:}}\hspace{3mm} Assume we have an integrable one-form $\alpha$ of the form
\begin{equation}\label{MCgen}
\alpha = A(y)\dd x + B(y) \dd y.
\end{equation}Comparing (\ref{MCgen}) to  (\ref{MC}) it is clear that we need to solve the equation $L(y)^{-1} \partial_y L (y) =B(y)$ with some function $L(y)$  into  $G$,
where $y$ varies in the interval $\St _y$ defining the strip $\St$ in $y-$direction.
It is possible to solve this differential equation on all of $\St _y$, since the equation is an ordinary differential equation and the domain of definition is an interval. Using such a solution  it is straighforward to verify  $\partial_y (L(y) A(y) L(y)^{-1} ) = 0$
in view of the integrability condition of $\alpha$.  Therefore,  $A(y)$ has the form as required. As a consequence, any solution to $F(x,y)^{-1} \dd F(x,y) = \alpha (x,y)$  in $G$ as well as $ \exp(xD)L(y)$ have the same Maurer-Cartan form, namely $\alpha$. As a consequence we obtain
\begin{equation}\label{solalpha1}
F(x,y) = W \exp(xD) L(y)
\end{equation}
for some $W \in G$.
But then
$
F(x,y) =  \exp(x WDW^{-1}) (W L(y))
$
shows that $F$ produces an equivariant map into $G/K$ by projection.
\end{proof}

\begin{remark}
The freedom in the choice of $W$ reflects the freedom to move around the image of $\St$ in $G/K$.
\end{remark}

%%%%%%%%%%%%%%%%%%%%%%%%%%%%%%

\section{TE primitive harmonic maps and their loop group formalism}

In this section we will collect the basic theory concerning
TE (translationally equivariant) primitive harmonic maps and their loop group formalism. Then we will show how it applies to Willmore surfaces.
%%%%%%%%%%%%%%%%%%%%%
\subsection{Basics about $k-$symmetric spaces}

We will restrict our attention to   primitive harmonic maps. \emph{ We   always assume that $G$ is a real semisimple Lie group and $K$ a closed subgroup  containing the center of $G$.
Assume furthermore  that  for the Lie group $G$ there exists an  automorphism $\sigma$ of $G$ of order $k$  such that $Fix_\sigma (G)^0 \subset K \subset Fix_\sigma (G)$, where $H^0$ denotes  the connected component of $H$ containing the identity element $e \in H$ for a Lie group $H$.
In this setting the space $G/K$ will be called a ``$k-$symmetric space''. \emph{Moreover, in what follows, we will always assume that $G/K$ carries a metric which is induced
from a bi-invariant metric on $G$. Also,  we will  primarily consider maps
$f:M \rightarrow G/K$ from a Riemann surface $M$ which are smooth and conformal  relative to a  standard metric on $M$ and the metric on $G/K$ induced from a bi-invariant metric on $G$.} We will also assume that $f:M \rightarrow G/K$ is harmonic with respect to the metrics mentioned above.
In this setting $M$ is  orientable. The case of maps from non-orientable surfaces  to $k-$symmetric spaces will be discussed in Section 6 (see also  \cite{DoWaSym2}).
}

 Let $\sigma:G\rightarrow G$ be the  automorphism of order $k\geq 2$ which defines $G/K$.
Let  $\mathfrak{g}$ be  the Lie algebra of $G$. Then $\sigma$ induces an  automorphism of order $k$ on $\mathfrak g$, also called $\sigma$. Moreover, $\sigma$ gives a $K-$stable
$\Z_k$ grading $\mathfrak g^{\C}=\sum_{j\in \mathbb{ Z}_k}\mathfrak g_j,$ where
$\mathfrak g_j$ is the $\omega^j-$eigenspace of $\sigma,$ where we set  $\omega=e^{2\pi i/k}$.
Set $\mathfrak m=\mathfrak g\cap (\sum_{j\in\mathbb {Z}_k,j\neq0}\mathfrak{g}_j)$.
Then we have the  the decomposition of the Lie algebra $\mathfrak g$:
\[\mathfrak{g} = \mathfrak{k} + \mathfrak{m}.\]
The crucial Lie algebra for the approach of this paper is the  $\sigma-$twisted Lie algebra $\Lambda \mathfrak{g}_{\sigma}$ is defined by
\begin{equation}
\Lambda \mathfrak{g}_\sigma = \sum_{j \in \mathbb{Z}} \lambda^{-j} X_j,
\end{equation}
where $X_j \in \mathfrak{g}_m$ with $m = j \mod k$ and $\lambda \in S^1$ and $\C^*$ respectively.

%%%%%%%%%%%%%%%%%%%
\subsection{Primitive harmonic maps}
Let $f:M \rightarrow G/K$ be a map from $M$ to $G/K$. Let $F$ denote any frame for $f$ and $\alpha$ its Maurer-Cartan form {\bf ( M-C form)} , $\alpha = F^{-1} \dd F$.
Then $\alpha$ decomposes in two ways: firstly, we can decompose the one-form $\alpha$ in the form $\alpha = \alpha^{(1,0)} + \alpha^{(0,1)}$, but, secondly,  also with regard to the eigenspace decomposition of $\sigma$  (by abuse of notation we will write $\sigma$ for $\dd \sigma$).

\begin{definition}
A harmonic map
$f:M \rightarrow G/K$ is  called  {\bf primitive harmonic,} if the Maurer-Cartan form
$\alpha = F^{-1} \dd F$ has the form
\begin{equation}
\alpha = \alpha_{-1} + \alpha_{0} + \alpha_{1}
\end{equation}
with $\alpha_m \in \mathfrak g_m.$ and for which
\begin{equation} \label{prim}
\alpha_\lambda  = F^{-1} \dd F = \lambda^{-1} \alpha^\prime_\mathfrak{m} + \alpha_\mathfrak{k}+  \lambda \alpha''_\mathfrak{m}
\end{equation}
is integrable for all $\lambda \in S^1$.
\end{definition}
\begin{remark}\
\begin{enumerate}\item If $k >2$, then the  condition (\ref{prim}) above is true if and only if it is true for $\lambda =1$.
\item If $k=2$, then the condition (\ref{prim})  above needs to be true for at least three values of $\C^*$
in order to be true for all $\lambda \in S^1$.
\item It is well known that  primitive harmonic maps are real analytic.
\end{enumerate}
\end{remark}

%%%%%%%%%%%%%%%%%%%
\subsection{The loop group formalism for primitive harmonic maps}
%%%%%%%%%%%%%%%%%%%

Let us recall from \cite{DPW}, \cite{Bu-Ki} the basic loop group formalism for harmonic maps into  $k-$symmetric spaces $G/K$.

We start from some map $f$ as above, with $M=\St$ being an open strip containing the real axis. Let's assume from here on that the map $f$ is a primitive harmonic
map into $G/K$. Note that in this case $G$ is a group of isometries of $G/K$, relative to the chosen, definite or indefinite, metric.

Let $f:M\rightarrow G/K$ be a map.  Choose a frame $F$ of $f$ with $F(0,0) = I$ and
 set \[F^{-1}\dd F=\alpha=\alpha'\dd z+\alpha''\dd\bar{z}=(\alpha_{\mathfrak{m}}'+\alpha_{\mathfrak{k}}')\dd z+(\alpha_{\mathfrak{m}}''+\alpha_{\mathfrak{k}}'')\dd\bar{z}.\]
 Introduce $\lambda\in S^1$ via the M-C form of $F$ as usual (see e.g. \cite{DPW}, \cite{Bu-Ki})
\begin{equation}\label{form1}
\alpha_{\lambda} =(\lambda^{-1}\alpha_{\mathfrak{m}}'+\alpha_{\mathfrak{k}}')\dd z+(\lambda\alpha_{\mathfrak{m}}''+\alpha_{\mathfrak{k}}'')\dd\bar{z}= \mathbb{A}_{\lambda} \dd u + \mathbb{B}_{\lambda} \dd v,
\end{equation}
By definition,  $f$ being primitive harmonic is equivalent to  that
 $\alpha_{\lambda}$ is  integrable for all $\lambda\in S^1$.

 The solution
 \[F(z,\bar z,\lambda)^{-1}\dd F(z,\bar z,\lambda)=\alpha_{\lambda}, \ F(0,0,\lambda)=\mathbf{e}\]
 is called the extended frame of $f$ normalized at $z=0$.
Here $\mathbf{e}$  is the Identity element of $G$ and $F(z,\bar z,\lambda)$ takes values in the twisted loop group $\Lambda G_{\sigma}$ \cite{DPW,Bu-Ki}:
\[\Lambda G^{\mathbb{C}}_{\sigma} =\{\gamma:S^1\rightarrow G^{\mathbb{C}}~|~ ,\
\sigma \gamma(\lambda)=\gamma(\omega\lambda),\lambda\in S^1  \},\]\[ \Lambda G_{\sigma}   =\{\gamma\in \Lambda G^{\mathbb{C}}_{\sigma}
|~ \gamma(\lambda)\in G, \hbox{for all}\ \lambda\in S^1 \}.\]
 Performing a Birkhoff decomposition \cite{DPW, Bu-Ki} of $F(z,\bar z,\lambda)$ near $z=0$, we obtain
\[F(z,\bar z,\lambda)=F_-(z,\lambda)F_+(z,\bar z,\lambda)\]
where \[F_-(z,\lambda)\in\Lambda^- G^{\C}_{\sigma}=
\{\gamma\in \Lambda G^{\mathbb{C}}_{\sigma}~
|~ \gamma \hbox{ extends holomorphically to } |\lambda|>1 \cup\{\infty\} \}\] is meromorphic in $z$ and $F_-(0,\lambda)=\mathbf e$, and
\[F_+(z,\bar z,\lambda)\in\Lambda^+ G^{\C}_{\sigma}=\{\gamma\in \Lambda G^{\mathbb{C}}_{\sigma}~
|~ \gamma \hbox{ extends holomorphically to the disk}\}.\]
Then the normalized potential $\eta$ of $f$ is defined to be the M-C form of $F_-(z,\lambda)$:
\begin{equation}
\eta:=F_-(z,\lambda)^{-1}\dd F_-(z,\lambda)=\lambda^{-1}\eta_{-1}\dd z,\end{equation}
with $\eta_{-1}$ taking values in $\mathfrak m^{\C}.$

Conversely, beginning with a normalized potential of the above form, solving the ODE $F_-(z,\lambda)^{-1}\dd F_-(z,\lambda)=\eta$,  $F_-(0,\lambda)=\mathbf e$, one obtain a meromorphic extended frame $F_-(z,\lambda)$. Performing an Iwasawa decomposition \cite{DPW, Bu-Ki} of $F_-(z,\lambda)$ near $z=0$, we obtain
\[F_-(z,\lambda)=\tilde F(z,\bar z,\lambda)\tilde F_+(z,\bar z,\lambda)\]
with $\tilde F(z,\bar z,\lambda)\in \Lambda G_{\sigma}$, $\tilde F(0,0,\lambda)=\mathbf e$ and $\tilde F_+(z,\bar z,\lambda)\in \Lambda ^+G^{\mathbb{C}}_{\sigma}$. Then $\tilde f(z,\bar z, \lambda)=\tilde F(z,\bar z,\lambda)\mod K$ is a family of primitive harmonic maps into $G/K$. This is the so called DPW method for primitive harmonic maps \cite{DPW, Bu-Ki}.

On the other hand, one can do a different type of  decomposition of $F(z,\bar z,\lambda)$ near $z=0$ (See e.g. \cite{DPW, Bu-Ki}):
\[F(z,\bar z,\lambda)=C(z, \lambda)\hat F_+(z,\bar z,\lambda),\]
such that
$C(z, \lambda)$ is holomorphic in $z$, $C(0,\lambda)=\mathbf e$, and $\hat F_+(z,\bar z,\lambda)\in\Lambda^+ G^{\C}_{\sigma}$. Then one obtains a holomorphic potential
which is  the M-C form of $C(z,\lambda)$:
\begin{equation}
\xi:=C(z,\lambda)^{-1}\dd C(z,\lambda)=\sum_{j=-1}^{\infty}\lambda^{j}\xi_{j}\dd z\in\Lambda\mathfrak{g}^{\C}_{\sigma}:=Lie(\Lambda{G}^{\C}_{\sigma}).\end{equation}

Conversely, beginning with a holomorphic potential of the above form, solving the ODE $C(z,\lambda)^{-1}\dd C(z,\lambda)=\xi$,  $C(0,\lambda)=\mathbf e$, one obtains a holomorphic extended frame $C(z,\lambda)$. Performing an Iwasawa decomposition \cite{DPW, Bu-Ki} of $C(z,\lambda)$ near $z=0$, we obtain
\[C(z,\lambda)=\tilde F(z,\bar z,\lambda)\tilde F_+(z,\bar z,\lambda)\]
with $\tilde F(z,\bar z,\lambda)\in \Lambda G_{\sigma}$, $\tilde F(0,0,\lambda)=\mathbf e$ and $\tilde F_+(z,\bar z,\lambda)\in \Lambda ^+G^{\mathbb{C}}_{\sigma}$. Then $\tilde f(z,\bar z, \lambda)=\tilde F(z,\bar z,\lambda)\mod K$ is a family of primitive harmonic maps into $G/K$.

%%%%%%%%%%%%%%%%%%%%%%%%%%%%%%%%%%%%
\subsection{{The loop group formalism for TE primitive harmonic maps}}
%%%%%%%%%%%%%%%%%%%

Now let us focus on the case of $f$ being a TE primitive harmonic map. Choose a frame $F$ of $f$ satisfying \eqref{equi-F} and  $F(0) = I$.
Clearly, $\alpha_{\lambda}$ only depends on $v$. Setting
\begin{equation}\label{form2}
\mathbb{A}_{\lambda}=\lambda^{-1}\alpha_{\mathfrak{m}}'+\alpha_{\mathfrak{k}}'+\alpha_{\mathfrak{k}}''+\lambda\alpha_{\mathfrak{m}}'',\ \hspace{3mm}
\mathbb{B}_{\lambda}=i\left(\lambda^{-1}\alpha_{\mathfrak{m}}'+\alpha_{\mathfrak{k}}'-\alpha_{\mathfrak{k}}''-\lambda\alpha_{\mathfrak{m}}''\right),
\end{equation}
we have
$\alpha_{\lambda} = \mathbb{A}_{\lambda} \dd u + \mathbb{B}_{\lambda} \dd v$ with $\mathbb{A}_{\lambda}$ and $\mathbb{B}_{\lambda} $ depending only on $v$.
Therefore from
$\alpha_{\lambda}$, we obtain an equivariant map $f(u,v,\lambda)$ with
moving frame
\begin{equation}
F(u,v,\lambda) = \exp( u  D(\lambda) ) L(v,\lambda),
\end{equation}
where $D(\lambda)\in\Lambda g_{\sigma}$ satisfies
$D(\lambda=1)=D$ and $ L(v,\lambda)$ is a solution to
\[L(v,\lambda)^{-1}\dd L(v,\lambda)=\mathbb{B}_{\lambda}\dd v,\ L(v,\lambda=1)=L(v),\]
and where we have assumed $F(z_0,\bar{z_0}, \lambda )= I$.  For the frame $F$ we have
\begin{equation*} \label{equi-F-loop}
F(z+t,\bar{z}+t,\lambda)= \exp( (u+t) D(\lambda)) L(v,\lambda)=
\chi_t(\lambda)F(z,\bar{z},\lambda)
\end{equation*}
with $\chi_t(\lambda)=e^{t D(\lambda)}$ being a one--parameter group in $\Lambda G_{\sigma}$.
 Setting $t=-u,z=u+iv$, we derive
\begin{equation}  \label{Iwadec}
\begin{split}F(z,\bar{z},\lambda)&=e^{uD(\lambda)}F(iv,-iv,\lambda)=e^{(u+iv)D(\lambda)}V(v,\lambda)\\
\end{split}
\end{equation}
with
\begin{equation*}
V(v,\lambda)=e^{-ivD(\lambda)}F(iv,-iv,\lambda)=
e^{-ivD(\lambda)}L(v,\lambda).
\end{equation*}
As in the case of Delaunay surfaces in $\R^3$ or in the case of primitive harmonic maps considered in \cite{Bu-Ki} we would like to consider
the differential form $\xi =\hat{ D}(\lambda) \dd z$ as a holomorphic potential for some equivariant primitive harmonic map into $G/K$. For this it suffices to show
$ V \in \Lambda^+G^{\mathbb{C}}_{\sigma}$. It will turn out to be very easy to verify this condition.
As a preliminary step we note the following equations, which follow directly from \eqref{Iwadec}:
\begin{equation}\label{prep1}
F^{-1} \partial_{\bar{z}} F= V^{-1} \partial_{\bar{z}} V = V^{-1}\frac{i}{2} \partial_{v}V,
\end{equation}
\begin{equation}\label{prep2}
V^{-1} \dd V =( -i L^{-1} DL  +
L ^{-1}\partial_v L) \dd v =(-i \mathbb{A}_{\lambda} + \mathbb{B}_{\lambda} )\dd v=
-2i \alpha'' \dd v.
\end{equation}

After these preparations we can prove

\begin{proposition} Using the notation above the following statements are equivalent
\begin{enumerate}
\item
The decomposition (\ref{Iwadec}) is an Iwasawa decomposition
for every $z \in \St$,

\item $V(v,\lambda) \in  \Lambda^+G^{\mathbb{C}}_{\sigma}$  for all $z=u+iv \in \St$.

\item $V(v_0,\lambda) \in  \Lambda^+G^{\mathbb{C}}_{\sigma}$
 for some $z_0 =u_0+iv_0\in \St$.

\item $V(0,\lambda)  \in K$ is indpendent of $\lambda$.

\item $F(0,0,\lambda) \in K$ is independent of $\lambda$.
\end{enumerate}
\end{proposition}

\begin{proof} The statements $(1)$ and $(2)$ are obviously equivalent.
Also $(2)$ {clearly implies  $(3)$.}

To prove the equivalence of $(2)$ and $(3)$ we use a trick (\cite{Bu-Ki}):
in view of  (\ref{prep2}) it  suffices to show that $V $ is contained in
$\Lambda^+\mathfrak{g}^{\mathbb{C}}_{\sigma}$ for some point $z_0 \in \St$, since it solves an ordinary differential equation which does not contain any negative power of $\lambda$.
But this is the assumption of $(3)$.

In view of \eqref{Iwadec} the statements $(4)$ and $(5)$ are equivalent.

Starting from $(3)$ we consider again the relation
\begin{equation}\label{eq-345}
\hbox{$F(z,\bar{z},\lambda)= \exp(z D(\lambda)) V(v,\lambda)$ with
 $V(v,\lambda) \in  \Lambda^+G^{\mathbb{C}}_{\sigma}$.}
 \end{equation} Setting $z=0$ we obtain $F(0,0,\lambda)  = V(0,\lambda)$, whence
 $(4)$ as well as $(5)$ holds. Conversely, if $(4)$ or $(5)$ are satisfied, then the relation in \eqref{eq-345} (which is always true in the equivariant case) implies  that in either case $V(0,\lambda)  \in K \subset \Lambda^+G^{\mathbb{C}}_{\sigma}$, from which $(3)$ follows.
\end{proof}

%%%%%%%%%%%%%%%%%%%%%%%%%%%%
\subsection{Relating $D(\lambda)$ to the Maurer-Cartan form of the extended frame}

Since we want   (\ref{Iwadec})  to be an Iwasawa decomposition, by the proposition we need to assume ${F}(0) \in K$ independent of $\lambda$.
Actually, from here on we will always assume ${F}(0) =I$ unless the opposite will be stated explicitly.

\begin{theorem} \label{thm-mcform} We retain the notation as above.
\begin{enumerate}
\item  Let $\St$ be a strip as above and let $F$ denote the extended frame of some TE
harmonic map $f: \St \rightarrow G/K$, satisfying
\begin{equation} \label{equi-F-loop-1}
\alpha_{\lambda}=F(z,\bar{z},\lambda)^{-1}\dd F(z,\bar{z},\lambda), ~~F(z+t,\bar{z}+t,\lambda)= e^{t D(\lambda)} F(z,\bar{z},\lambda).
\end{equation}
Then
\begin{equation} \label{eq-defD2}
D(\lambda) =\alpha_{\lambda}|_{z=0}=\alpha'_{\lambda} (0) + \alpha''_{\lambda} (0) =\lambda^{-1} \alpha_{\mathfrak{m}}'(0)+ \alpha_{\mathfrak{k}}'(0)+\alpha_{\mathfrak{k}}''(0) + \lambda \alpha_{\mathfrak{m}}''(0)\in\Lambda\mathfrak{g}_{\sigma}.
\end{equation}
Moreover, set
\begin{equation}  \label{eq-defV}
V(v,\lambda) = e^{-iv\cdot D(\lambda) } \cdot F(iv,-iv,\lambda),
\end{equation}
where $v=-\frac{i}{2}(z+\bar{z})$. Then $V(v,\lambda)$ takes values in $\Lambda^+ G^{\C}_{\sigma}$ and satisfies
\begin{equation}\label{eq-mc1} V'=V_v=-2i V\alpha_{\lambda}'',\ ~ V(0,\lambda)=I, ~~\hbox{ and }~\alpha_{\lambda}'+\alpha_{\lambda}''=V^{-1} D(\lambda) V.
\end{equation}
Note that this also  shows that  $F(z,\bar{z},\lambda)=e^{z D(\lambda) }V(y,\lambda)$ provides an  Iwasawa decomposition  of $e^{zD(\lambda) }$.
\item Conversely, let $D(\lambda)\in\Lambda\mathfrak{g}_{\sigma}$ such that $D(\lambda)=\lambda^{-1} D_{-1} + D_0 + \lambda D_1$, with $D_0\in\mathfrak{k}$, $D_{-1},D_{1}\in\mathfrak{m}^{\C}$. Let $V(v,\lambda):(-\varepsilon,\varepsilon)\rightarrow \Lambda^+ G^{\C}_{\sigma}$ be a map such that
         \begin{equation}\label{eq-Iwadec1}
         \left\{\begin{split}
         & ~V_v=-2i V\widetilde{\alpha}_{\lambda}'' , ~ V(0,\lambda)=I,\\
         & ~V^{-1} D(\lambda) V =\overline{\widetilde{\alpha}_{\lambda}''}+\widetilde{\alpha}_{\lambda}''.
         \end{split}     \right.
       \end{equation}
Then
\begin{equation}\label{eq-Iwadec3}
\widetilde{F}(z,\bar{z},\lambda) = e^{z D(\lambda)} \cdot V(v,\lambda)
\end{equation}
is  an Iwasawa decomposition
 of $e^{z D(\lambda) }$ and hence,
$\tilde{f}(z,\bar{z},\lambda)=\widetilde{F}(z,\bar{z},\lambda)\mod K$,
is a TE harmonic map into $G/K$. Moreover, we have
\begin{equation}\label{eq-Iwadec4}
\widetilde{F}(z,\bar{z},\lambda)^{-1} \dd\widetilde{F}(z,\bar{z},\lambda)=\widetilde{\alpha}_{\lambda}'\dd z+\widetilde{\alpha}_{\lambda}''\dd\bar{z},
\hbox{ with } \widetilde{\alpha}_{\lambda}'=\overline{\widetilde{\alpha}_{\lambda}''}=V^{-1} D(\lambda) V-\frac{i}{2}V^{-1}V_v.
\end{equation}

\end{enumerate}
\end{theorem}
\begin{proof}

(1). By \eqref{eq-defV}, we have that $V(0,\lambda)=I$ and  $F(z,\bar{z},\lambda)=e^{z D(\lambda) }V(v,\lambda)$. So
\[\alpha_{\lambda}=F(z,\bar{z},\lambda)^{-1}\dd F(z,\bar{z},\lambda)=V^{-1}D(\lambda) V(v,\lambda)\dd z+V^{-1}V_v\dd v.\]
As a consequence,  \eqref{eq-defD2} and \eqref{eq-mc1} follow from this equation. Since $\alpha_{\lambda}''\in\Lambda^+\mathfrak{g}^{\C}$ and $V(0,\lambda)=I$, we see by \eqref{eq-mc1}
that  $V(v,\lambda)$ takes values in $\Lambda^+ G^{\C}_{\sigma}$, which means that $F(z,\bar{z},\lambda)=e^{z D(\lambda) }V(v,\lambda)$ provides an Iwasawa decomposition  of $e^{zD(\lambda) }$.

(2). From $V^{-1} D(\lambda) V =\overline{\widetilde{\alpha}_{\lambda}''}+\widetilde{\alpha}_{\lambda}''$, we see that $\lambda^{-1}\widetilde{\alpha}_{\lambda}''$ takes values in $\Lambda ^-  \mathfrak{g}^{\C}$ since the only negative power of $\lambda$ in the right side is $\lambda^{-1}$.
Moreover, from \eqref{eq-Iwadec4} and \eqref{eq-Iwadec3}, we have
\[\widetilde{F}(z,\bar{z},\lambda)^{-1} \dd\widetilde{F}(z,\bar{z},\lambda)=V^{-1} D(\lambda) V\dd z+V^{-1}V_v=\widetilde{\alpha}_{\lambda}'\dd z+\widetilde{\alpha}_{\lambda}''\dd\bar{z},~~ \widetilde{F}(0,0,\lambda)=I.\]
Since $\overline{\widetilde{\alpha}_{\lambda}'+\widetilde{\alpha}_{\lambda}''}=\widetilde{\alpha}_{\lambda}'+\widetilde{\alpha}_{\lambda}''$, we see that $\widetilde{F}(z,\bar{z},\lambda)$ takes values in $\Lambda G_{\sigma}$ and that $\tilde{f}(z,\bar{z},\lambda)=\widetilde{F}(z,\bar{z},\lambda)\mod K$ is a TE harmonic map.
\end{proof}

\begin{remark}
\
 \begin{enumerate}
\item The potential $D(\lambda)$ in \eqref{eq-defD2} is called a Delaunay potential.
\item From \eqref{eq-defD2} it is clear that $D_0$ can be simplified to the same extent as $\alpha_{\mathfrak{k}}(0)=\alpha_{\mathfrak{k}}'(0)+\alpha_{\mathfrak{k}}''(0)$ can be.
For  several other surface classes, like CMC surfaces in $\R^3$ or minimal Lagrangian surfaces in $\C P^2$ one can  remove $D_0$ completely.
In the case of Willmore surfaces this is not possible by the construction of the  potential $D(\lambda)$ (see e.g. \eqref{eq-TE-Willmore-S3})
\item We have seen above that to find a global Iwasawa splitting $\widetilde{F}(z,\bar{z},\lambda) = e^{z D(\lambda)} \cdot V(v,\lambda)$, we need to find $V \in \Lambda^+G^{\mathbb{C}}_{\sigma}$
such that the equations \eqref{eq-Iwadec1} are satisfied globally.
\end{enumerate}
\end{remark}

%%%%%%%%%%%%%%%%%%%%%%%
%%%%%%%%%%%%%%%%%%%%
\subsection{About the domain of definition of TE primitive harmonic maps}
%%%%%%%%%%%%%%%%%%%%%%%%

While the potential of a TE primitive harmonic map is obviously defined on $\C$, the actual surface is generally only an immersion on some (small) strip. This is a consequence of the fact that the Iwasawa decomposition for non-compact real groups $G$ is not global.

\begin{theorem}
	Let $\xi = D(\lambda) dz$ be a Delaunay potential.
	Then
	
	$(1)$  $C(z,\lambda) = e^{zD(\lambda)}$
	is a holomorphic extended frame and the extended frame $F(z, \bar{z}, \lambda)$ is defined on some open strip $\St=\{z\in\C| Re(z)\in(u_1,u_2)\}$ with $-\infty\leq u_1<0<u_2\leq+\infty$. Moreover,  the corresponding harmonic map $f(z, \bar{z}, \lambda)_{\lambda=1}=F(z, \bar{z}, \lambda)_{\lambda=1}\mod K$ is TE on  $\St$.
	
	$(2)$ If $C(z,\lambda = 1 ) $ is periodic with some real period $p$, then the harmonic map $f$  descends to a RE harmonic map on the annulus obtained as
	image of $\St$  by taking the quotient $\C\mod p\mathbb{Z}$.
\end{theorem}

Since the Iwasawa decomposition for  $\Lambda G^{\C}_{\sigma}$ is  global when $G$ is compact, we have
\begin{theorem}
	Assume that $G$ is compact. Then
	\begin{enumerate}
		\item Every TE harmonic map into any $k-$symmetric space $G/K$ is defined and smooth on all of $\C$.
		
		\item Conversely, let $\xi = D(\lambda) dz$ be a Delaunay potential.
		Then $C(z,\lambda) = e^{zD(\lambda)}$
		is a holomorphic extended frame and the extended frame $F(z, \bar{z}, \lambda)$ is defined on $\C$ and the corresponding harmonic map  $f(z, \bar{z}, \lambda=1)$ is TE on  $\C$.
		Moreover, if $C(z,\lambda = 1 ) $ is periodic with some real period, then the harmonic map  $f(z, \bar{z}, \lambda=1)$  descends to a RE harmonic map on $\C^*$  by taking the quotient $ \C \mod p\mathbb{Z}$.
	\end{enumerate}
\end{theorem}

%%%%%%%%%%%%%%%%%%%%%%%

%%%%%%%%%%%%%%%%%%%%%%%%%%%
\section{TE Willmore surfaces}
In this section, we will present as examples and applications the holomorphic potentials of  TE Willmore surfaces and discuss the domain where the corresponding TE Willmore surface is immersed.
  We refer to \cite{DoWa10} for the detailed discussion of Willmore surfaces via loop groups.

\subsection{Holomorphic potentials of TE Willmore surfaces}
Here we first recall that a Willmore surface $y$ always has (in a one-one correspondence ) an oriented conformal Gauss map $Gr$ into
$Gr_{1,3}\R^{n+4}_{1}=SO^+(1,n+3)/SO^+(1,3)\times SO(n)$.
We can choose a frame
\begin{equation}\label{eq-F}
F=\left(\frac{1}{\sqrt{2}}(Y+N),\frac{1}{\sqrt{2}}(-Y+N),e_1,e_2,\psi_1,\cdots,\psi_n\right)
\end{equation}
of $Gr$ which attains values in $SO^+(1,n+3)$
where $Y$ is a ``canonical lift'' of $y$ w.r.t to some complex coordinate $z$, and
$\{\psi_j,j=1,\cdots,n\}$ is a basis of the conformal normal bundle of $Y$.

The Maurer-Cartan form $\alpha=F^{-1}\dd F$ of $F$ is of the form \[\alpha=\left(
                   \begin{array}{cc}
                     A_1 & B_1 \\
                     B_2 & A_2 \\
                   \end{array}
                 \right)\dd z+\left(
                   \begin{array}{cc}
                     \bar{A}_1 & \bar{B}_1 \\
                     \bar{B}_2 & \bar{A}_2 \\
                   \end{array}
                 \right)\dd \bar{z},\]
with
\begin{equation}A_1=\left(
                             \begin{array}{cccc}
                               0 & 0 & s_1 & s_2\\
                               0 & 0 & s_3 & s_4 \\
                               s_1 & -s_3 & 0 & 0 \\
                               s_2 & -s_4 & 0 & 0 \\
                             \end{array}
                           \right),\   A_2=\left(
                             \begin{array}{cccc}
                               b_{11} & \cdots &  b_{n1} \\
                               \vdots& \vdots & \vdots \\
                               b_{1n} &\cdots & b_{nn} \\
                             \end{array}
                           \right),\ \end{equation}
\begin{equation}\label{s}
\left\{\begin{split}&s_1=\frac{1}{2\sqrt{2}}(1-s-2k^2),\ s_2=-\frac{i}{2\sqrt{2}}(1+s-2k^2),\\
&s_3=\frac{1}{2\sqrt{2}}(1+s+2k^2),\ s_4=-\frac{i}{2\sqrt{2}}(1-s+2k^2),\\
\end{split}\right.
\end{equation}
\begin{equation} \label{B1}
B_1=\left(
      \begin{array}{ccc}
         \sqrt{2} \beta_1 & \cdots & \sqrt{2}\beta_n \\
         -\sqrt{2} \beta_1 & \cdots & -\sqrt{2}\beta_n \\
        -k_1 & \cdots & -k_n \\
        -ik_1 & \cdots & -ik_n \\
      \end{array}
    \right),  \ \
B_2=\left(
      \begin{array}{cccc}
        \sqrt{2} \beta_1& \sqrt{2} \beta_1 & k_1& i k_1 \\
        \vdots & \vdots & \vdots & \vdots\\
      \sqrt{2} \beta_n& \sqrt{2} \beta_n & k_n & ik_n \\
      \end{array}
    \right)=-B_1^tI_{1,3}.
 \end{equation}
Here (see \cite{BPP,DoWa12} for more details)
\[Y_{zz}=-\frac{s}{2}Y+\kappa.\]
Moreover, choosing an oriented orthonormal frame $\{\psi_j,j=1,\cdots,n\}$ of the normal bundle $V^{\perp}$ over $U$,
we can write the normal connection as $D_z\psi_j=\sum_{l=1}^{n}b_{jl}\psi_l$ with $b_{jl}+b_{lj}=0.$
Then, the conformal Hopf differential $\kappa$ and its derivative $D_{\bar{z}}\kappa$ are of the form
\begin{equation} \label{defkappabeta}
\kappa=\sum_{j=1}^{n}k_j\psi_j,\ D_{\bar{z}}\kappa=\sum_{j=1}^{n}\beta_j\psi_j,\
\hbox{with }\beta_j=k_{j\bar{z}}- \sum_{j=1}^{n}\bar{b}_{jl}k_l,\ j=1,\cdots,n.\end{equation}
We also have $k=\sqrt{\sum_{j=1}^n|k_j|^2}$   and
$I_{1,3}=\hbox{diag}\{-1,1,1,1\}$. Note that $b_{ij}=\langle \psi_{iz},\psi_j\rangle$ gives the normal connection.

Applying this  to TE Willmore surfaces we obtain, in view \eqref{eq-defD2} and  \cite[Proposition 2.2]{DoWa12}, the following result.
\begin{proposition}\label{cor-TEW-S3} For a TE Willmore surface in $S^{n+2}$, its holomorphic potential
\begin{equation}\label{eq-TE-Willmore-S3}
\xi =D(\lambda)\dd z=\left(
                   \begin{array}{cc}
                     \tilde{A}_1& B_1(\lambda) \\
                     -(B_1(\lambda))^tI_{1,3} & \tilde{A}_2 \\
                   \end{array}
                 \right)\dd z,
\end{equation}
with
\[\tilde A_1=\left.\left(
                             \begin{array}{ccccc}
                               0 & 0 & \frac{1-Re(s)-2|k|^2}{\sqrt{2}} &  \frac{Im(s)}{\sqrt{2}}  \\
                               0 & 0 & \frac{1+Re(s)+2|k|^2}{\sqrt{2}} &- \frac{Im(s)}{\sqrt{2}}  \\
                               \frac{1-Re(s)-2|k|^2}{\sqrt{2}} & -\frac{1+Re(s)+2k^2}{\sqrt{2}} & 0 & 0   \\
                                \frac{Im(s)}{\sqrt{2}} &  \frac{Im(s)}{\sqrt{2}} & 0 & 0  \\
                             \end{array}
                           \right)\right|_{z=0},\ \tilde A_2= \left.\left(A_2+\overline{A_2}\right)\right|_{z=0},\]
                           and
                           \[
                           B_1(\lambda)=\left.\left(\lambda^{-1}B_1+\lambda \overline{B_1}\right)\right|_{z=0}.\]
\end{proposition}

\begin{remark}
It is important to observe that the form of $A_1$ implies  that the ``Condition a)'' of
Theorem 3.11 of \cite{DoWa10} is satisfied. This implies in the TE case that in some neighbourhood of the base point $0$ the surface is an immersion, which implies that there is a strip $\St$ containing the real axis on which the surface constructed from $\xi$ actually is an immersion.
We will discuss  in Section \ref{S-imme} the issue of what singularities the surface derived from $\xi$ has.
We will show that this surfaces does not have any branch points, i.e. points, where the surfaces is  smooth, but has a differential of rank $<2$.
\end{remark}

\subsection{ The immersion property of TE Willmore maps on strips}\label{S-imme}

In this subsection we will show that for the TE harmonic map defined on a strip $\mathbb{S}$ corresponding to a Willmore map $y$, the Willmore immersion  $y$ will be  defined on the strip $\St$.

The discussion of this subsection is applicable to many other classes of ``integrable surfaces". We restrict attention to Wilmore surfaces since this will not require to introduce new mathematical machinery.

\begin{theorem}\label{nobranchWill}
  Let $y: \mathbb{S}\rightarrow S^{n+2}$ be a conformal, real analytic, TE Willmore map defined on some strip $\St$. Then it does not have any branch points on $\St$. In particular, $y$ does not have any branch points if it is periodic.
\end{theorem}

\begin{proof}
Let $y: \mathbb{S}\rightarrow S^{n+2}$  be the TE Willmore map, then
\[y(u,v)=R(u)y(0,v), \]
$\hbox{with } R(u):\R\rightarrow SO^+(1,n+3)$ being the 1-parameter group. Assume w.l.g. that $(0,0)\in \St$. Suppose  that $(0,0)$ is a branch point of $y$. Then $(u,0)$ is a branch line. Since $y$ and $y_u$ are real analytic on $\St$ and $y$ is conformal, we have  $y_u(u,0)=y_v(u,0)=0$ for all $u\in\R$. Moreover, since $y(u,v)$ is real analytic, we have
 \[y(u,v)=y_0(u)+\sum_{j=1}^{\infty}y_j(u)v^j\]
for some $y_0(u)$, $y_j(u)$ defined on $\St$. So $y_u(u,v)=y_0'(u)+\sum_{j=1}^{\infty}y_j'(u)v^j$, where $y_j'(u)=\frac{d y_j(u)}{du}$ for $0\leq j<\infty$. Substituting  $y_u(u,0)=0$, we obtain
$y_0'(u)=y_u(u,0)=0$, which implies that $y_0(u)=p_0\in S^{n+2}$ is constant. So
 \[y(u,v)=p_0+\sum_{j=l}^{\infty}y_j(u)v^j, \hbox{ with } y_l(u)\not\equiv0.\]

Choose $u_0\in\R$ such that $y_l(u_0)\neq0$. As a consequence,  we have near $(u_0,0)$
\[||y_u(u_0,v)||^2= ||y_l'(u_0)||^2v^{2l} {+o(v^{2l +1}),}\ ||y_v(u_0,v)||^2= ||ly_l(u_0)||^2v^{2l-2}+o(v^{2l-1}), \]
contradicting  the conformality condition $||y_u(u,v)||^2=||y_v(u,v)||^2$ on the whole $\St$,
since the two expressions converge to $0$ for $v \rightarrow 0$  with different speeds.
\end{proof}

From the definition of an associated family, we obtain immediately
 \begin{corollary} No member of the associated associated family of a Willmore surface as considered in Theorem \ref{nobranchWill}, does have any  branch point.
\end{corollary}

\begin{remark}\
\begin{enumerate}
\item The theorem above means that if the Willmore map is defined on $\mathbb S$, it is an immersion. But it does not tell how large the strip $\mathbb S$ will be.

\item In the theorem above we have asserted that  TE Willmore surface defined on some strip does not have any branch points. However, while the original potential for a TE Willmore surface is defined on all of $\C$, smooth solutions may only exist on (perhaps infinitely many different ) strips.
The basic reason is that starting from the potential $\xi = D(\lambda)dz$ one needs to perform an Iwasawa splitting. Since the group $G$ in $G/K$ is non-compact in the Willmore case, the Iwasawa splitting will not be defined globally, equivalently, the coefficients of the frame will become  singular along some lines (by the TE property).
\end{enumerate}
\end{remark}

%%%%%%%%%%%%%%%%%%%%%%%%%

\section{TE primitive harmonic maps with real and with non-real periods} \label{ss-cylinder}

In this section we will first consider separately TE primitive harmonic maps with real and with non-real
periods respectively. Correspondingly we will discuss TE Willmore cylinders with real and with
non-real periods respectively.
%%%%%%%%%%%%%%%%%%%%%%%%%%
\subsection{TE primitive harmonic cylinders and TE Willmore cylinders}
Let's start with a basic definition.
\begin{definition}\ \begin{enumerate}
\item Let $f:\C\rightarrow G/K$ be a TE primitive harmonic map. Assume there exists some $\omega \in \C$ such that $f(z+\omega) = f(z)$ for all $z \in \C $, then the map
\begin{equation}
f_\omega : \C/ {\omega \mathbb{Z} } \rightarrow  G/K,~~f_\omega (z \mod \omega \mathbb{Z} ):= f(z)
\end{equation}
is called a {\bf TE primitive harmonic cylinder}.

\item
Let  $y:\C \rightarrow S^{n+2}$ be a TE Willmore surface in $S^{n+2}$, with a conformal Gauss map $f:\C \rightarrow  G/K = SO^+ (1, n+3)/ {SO^+(1,3) \times SO(n)}$.
Assume there exists some $\omega \in \C$ such that $y(z+\omega) = y(z)$ for all $z \in \C $, then the map
\begin{equation}
y_\omega : \C/ {\omega \mathbb{Z} } \rightarrow  S^{n+2},~~y_\omega (z \mod \omega \mathbb{Z} ):= y(z)
\end{equation}
is called a {\bf TE Willmore cylinder}.
\end{enumerate}
\end{definition}

Recall that for a CMC surface in $\R^3$  it can happen that  its Gauss map is a TE harmonic cylinder with period $\omega$, but that the surface itself is not a TE CMC cylinder with  period  $\omega$. {See the last section of \cite{Bu-Ki} for a reference.}
In contrast to this, for Willmore surfaces we have

\begin{theorem}\label{lemma-sym}
$y$ is a TE Willmore cylinder with real period  $\omega$ if and only if its oriented conformal Gauss map $f$ is a TE harmonic cylinder with the same real period  $\omega$.
\end{theorem}
\begin{proof}
{ A Willmore surface and its oriented conformal Gauss map are in one to one correspondence \cite{Ma}. As a consequence $y$ has a real period  $\omega$ if and only if its oriented conformal Gauss map $f$ has a real period $\omega$.}
\end{proof}

%%%%%%%%%%%%%%%%%%%%%%%
\subsubsection{TE harmonic cylinders with real period}
%%%%%%%%%%%%%%%%%%%%%%%%

The simplest  choice of a period certainly is a real period. Let' recall here that we have assumed all surfaces to be full. In particular we have assumed that the center of $G$ is contained in $K$.

\begin{theorem}
Consider the  primitive harmonic map $f: \St \rightarrow G/K$. Then with the notation introduced above, a real number $p$ induces for $\lambda = 1$ a TE primitive harmonic cylinder
$f_p : \St/ {p \mathbb{Z} } \rightarrow  G/K,$ if and only if  $\exp( p D(\lambda)|_{\lambda = 1})$ lies in the center of $G$ (which is contained in $K$).
\end{theorem}
\begin{proof}
It suffices to point out that  the monodromy matrix for a real translation t is exactly the real matrix $\exp( t D(\lambda )) $ from which the claim follows.
\end{proof}

Applying this to Willmore surfaces we obtain
 \begin{corollary}
With the notation introduced above, in view of Proposition 2.2 of \cite{DoWa12}, let $D(\lambda)$ be of the form in Proposition 2.2 of \cite{DoWa12}. Then  a real number $p$ induces for $\lambda = 1$ a TE Willmore cylinder
$y_p : \St / {p \mathbb{Z} } \rightarrow  S^{n+2},$ if and only if  $\exp( p D(\lambda)|_{\lambda = 1}) = I$ if and only if the eigenvalues of $D(\lambda = 1)$ are integer multiples of $ \frac{2 \pi}{p}$.
\end{corollary}

This follows from the last result, since the center of $SO^+(1, n+3)$ is trivial.

We have pointed out above that due to the non-global Iwasawa decoposition TE Willmore surfaces may only be defined on some strips. In the discussion below we thus always consider the strip containing the origin.
 \begin{proposition}
 Set
 \[k=\frac{e^{i\theta}}{\sqrt{2}},\ \beta=0, ~ s=2c i~ \hbox{ with } c\in\R, \]
 in \eqref{eq-TE-Willmore-S3}.
So
\[
D(\lambda)=\left(
                             \begin{array}{ccccc}
                               0 & 0 &0 &  \sqrt{2} c & 0 \\
                               0 & 0 & \sqrt{2} &-   \sqrt{2} c &0  \\
                               0 & -\sqrt{2} & 0 & 0& -\frac{\lambda^{-1}e^{i\theta}+\lambda e^{-i\theta} }{\sqrt{2}} \\
                                 \sqrt{2} c &    \sqrt{2} c & 0 & 0 & -\frac{i(\lambda^{-1}e^{i\theta}-\lambda e^{-i\theta}) }{\sqrt{2}} \\
                               0 & 0 & \frac{\lambda^{-1}e^{i\theta}+\lambda e^{-i\theta} }{\sqrt{2}} &  \frac{i(\lambda^{-1}e^{i\theta}-\lambda e^{-i\theta}) }{\sqrt{2}} &0 \\
                             \end{array}
                           \right)
                           \]
 By straightforward computations we obtain that the eigenvalues of $D(\lambda)|_{\lambda=1}$ are
 \[0, \pm\sqrt{-2\pm2\sqrt{1-2c\sin\theta\cos\theta-\sin^2\theta+c^2}}.\]
\begin{enumerate}
\item Set $\sin\theta=1$. Then $D(\lambda)$ produces some TE Willmore cylinders with real period $p$ if and only if
\[ c=\frac{m^2-l^2}{m^2+l^2}, m,l\in\mathbb{Z}^+.\]
Moreover, in this case we have $p=\pi\sqrt{m^2+l^2}$. Note that in the case $c=\pm1$, we do not have $\exp( p D(\lambda)|_{\lambda = 1}) = I$ for any $p\in\R$.
\item Set $c=0$. We see that the eigenvalues become
 \[0, \pm\sqrt{-2\pm2\cos\theta}.\]
So, there exist infinitely many TE Willmore cylinders (with different real periods) in the associated family produced by $D(\lambda)$.
\end{enumerate}
\end{proposition}

\begin{corollary} Let $y$ be a TE Willmore cylinder in $S^3$.  Then
 there exist infinitely many TE Willmore cylinders (with different real periods) in the associated family of $y$.
\end{corollary}

A similar result holds for CMC surfaces \cite{Bu-Ki} and in particular for minimal Lagrangian surfaces in $\C P^2$, \cite{DoMaNL}, Section 6.

%%%%%%%%%%%%%%%%%%%%%%%
\subsection{TE harmonic  cylinders with non--real period}
%%%%%%%%%%%%%%%%%%%%%%%%

Now let's consider the  case of a translational period  $\omega$, where $\omega = \omega_1 + i \omega_2$ is not a real number, i.e. $\omega_2 \neq 0$.

\subsubsection{Some basic facts}

Before investigating this case in more detail we would like to point out

\begin{proposition}
Each TE primitive harmonic map $f: \St \rightarrow G/K$ which has an additional symmetry of the type $t_{\omega} = \{z \rightarrow {z +\omega} \}$ with  some non-real complex number $\omega = \omega_1 + i \omega_2$ is defined on $\C$ and has a quasi-invariant extended frame relative to $t_{\omega} $. Moreover, the Maurer-Cartan form of such a frame is periodic with period $\omega_2$.
\end{proposition}

\begin{proof}
Let's start from some quasi-invariant frame for the TE primitive harmonic map, see
Theorem \ref{thm-mcform}.
Now the fact that $f$ is invariant under $t_\omega$ implies that we have in addition to the relation $F(x, y, \lambda ) = e^{xD(\lambda)} F(0,y,\lambda)$  also the relation
 $F(z + i \omega_2, \lambda) = \chi(\lambda) F(z,\lambda) k(\omega_2,y)$ with
 $ \chi(\lambda) \in G$ for all $\lambda \in S^1$. {Of course, $f$ is defined on $\C$.
 Moreover, the} last factor is a cocycle relative to the action of the
group $\omega_2\mathbb{Z}$ on the strip $\St$. Since the real analytic cocycle
extends to a holomorphic cocycle on a complex strip including the $y-$axis, it does split, i.e. $k$ can be written in the form
$ k(\omega_2,y) = h(y) h(y+\omega_2)^{-1},$ and for $\hat{F} = F h$ the first claim follows.
The second claim is an immediate consequence of the first one.
\end{proof}

\begin{corollary} \label{period-alpha}
Let $\alpha = F^{-1} \dd F$ be the Maurer-Cartan form of a quasi-invariant frame as in the last Proposition. Then the set of purely imaginary periods of $\alpha$ is a lattice of the form
$p_2 i \Z$  with some real number $p_2 >0$ and $\omega_2$ is of the form $\omega_2 = m p_2$ with an integer $m$ which we can assume w.l.g. to be positive.
\end{corollary}

%%%%%%%%%%%%%%%%%%%%%%%%%%%
\subsubsection{A loop group theoretic characterization of non-real periods}

Let, again, $\omega$ be a non-real period of the TE primitive harmonic map $f$. Then we have at one hand for the holomorphic extended frame $C(z,\lambda) = e^{zD(\lambda)}$ the relation:
 \begin{equation}\label{eq-omega1}
C(z + \omega, \lambda) =
\exp( \omega D(\lambda)) C(z,\lambda).
\end{equation}
Since in this case the matrix $\exp( \omega D(\lambda))$ is not real, it is not a useful monodromy matrix in our theory (since for symmetries, and, in particular, for elements of some  fundamental group, the
monodromy matrices need to be contained in  $G$).

This apparent contradiction has a simple resolution (for CMC-surfaces in $\R^3$ see e.g. \cite{Do-Ha4}).

\begin{proposition}\label{prop-non-real}
Assume we have a TE harmonic map $f$ generated by $D(\lambda)$ and a complex number
$\omega$. Set $C(z,\lambda)=\exp(zD(\lambda))$ as above.
\begin{enumerate}
\item  If  the translation by $\omega$ is a symmetry of $f$, then there exists
some $\chi(\lambda)\in\Lambda G_{\sigma}$  and some
$V_+: \C \rightarrow \Lambda^+ G^\C_\sigma$,
such that
\begin{equation}\label{eq-chi}
C(z + \omega, \lambda) =
\chi(\lambda) C(z,\lambda) V_+(z,\lambda).\end{equation}
Setting  $b_+(\lambda)=V_+(0,\lambda)=\chi(\lambda)^{-1}\exp(\omega D(\lambda))\in\Lambda^+G^{\C}_{\sigma}$, we have
\begin{equation}\label{eq-b+}
b_+(\lambda) C(z,\lambda) = C(z,\lambda) V_+(z,\lambda).
\end{equation}
\item Conversely, if there
exists some  $b_+(\lambda)\in \Lambda^+G^{\C}_{\sigma}$ satisfying \eqref{eq-b+} such that
\[\chi(\lambda) =\exp(\omega D(\lambda))  b_+(\lambda)^{-1}\] is real, i.e. contained in
$\Lambda G_{\sigma}$, then the translation by $\omega$ is a symmetry of $f$
induced via  \eqref{eq-chi}.
\end{enumerate}

Moreover, if   $\chi(\lambda)_{\lambda=1}$ is contained in the center of $G$, then $f$ is a TE  harmonic cylinder with ``period'' $\omega$.
\end{proposition}
\begin{proof}
(1).
Comparing \eqref{eq-omega1} to \eqref{eq-chi} we derive:
\begin{equation*}
\chi(\lambda)^{-1} \exp( \omega D(\lambda)) C(z,\lambda) = C(z,\lambda) V_+(z,\lambda).\end{equation*}
Setting $z =0$ we obtain
\begin{equation*}
\chi(\lambda)^{-1} \exp(\omega D(\lambda)) = V_+(0,\lambda).
\end{equation*}

Now put  $b_+(\lambda) =  V_+(0,\lambda)$. Then  this yields
$b_+(\lambda) C(z,\lambda) = C(z,\lambda) V_+(z,\lambda)$. Note that this equation states equivalently that $b_+$ is an element of the isotropy group of
$C$ with regard to the dressing action.

(2) Conversely, assume we have a TE harmonic map and we are able to find some $b_+$ such that the last  equation holds and that also  $\chi(\lambda) =\exp(\omega D(\lambda))  b_+(\lambda)^{-1}$ is real. Then the translation by $\omega$ induces a  symmetry \eqref{eq-chi}.

The last statement comes from the fact that if we want to obtain a cylinder for $\lambda =1$, then the real monodromy matrix $\chi(\lambda)$ needs to be in the center of $G$ for $\lambda=1$.
\end{proof}

In order to understand which (non--real) $\omega$'s produce harmonic  cylinders,
we need to understand what $b_+$ there exist with the properties listed above.

\begin{theorem} \label{secondperiod}
Let $f_{\lambda}$ be the associated family of a full, TE primitive harmonic map generated by $D(\lambda)$  and defined on the strip $\St$.
Then each  $f_\lambda$ admits translation by $\omega$ as a symmetry  if and only if
 there exists some $b_+(\lambda)  \in \Lambda^+ G^\C_\sigma$ which commutes with $D(\lambda)$ and such that $\exp(\omega D(\lambda)) b_+(\lambda)^{-1} \in \Lambda G_\sigma.$

If, in addition,
$\chi(\lambda) = \exp(\omega D(\lambda)) b_+(\lambda)^{-1} $ is for $\lambda = \lambda_0$
in the center of $G$,  then $\omega$ is a period of $f_{\lambda_0} = f_{\lambda}|_{\lambda = \lambda_0}$ and $f_{\lambda_0}$ descends to a primitive harmonic map from the cylinder $\C / {\omega \Z}$ to $G/K$.
Moreover, the real translations induce a one-parameter group for which $f_{\lambda_0}$ is equivariant.
\end{theorem}

\begin{remark}\ \begin{enumerate}
\item Writing \[\chi(\lambda) = \exp(\omega D(\lambda)) b_+(\lambda)^{-1}
= \exp(\omega _1 D(\lambda))  \left (\exp(\omega_2i D(\lambda)) b_+(\lambda)^{-1} \right)\]
shows that  $\chi(\lambda)$ is a product of two real factors. Let's assume that $\omega$ is neither real nor purely imaginary. If both factors are in the center of $G$ for $\lambda = \lambda_0$, then one obtains a harmonic torus with rectangular fundamental parallelogram. If neither factor is in the center of $G$, then the map $f$ only closes after translation  by $\omega$.
\item For the case of Willmore surfaces, in both cases
above, by  Theorem \ref{lemma-sym},  the Willmore surface $y$ has the same symmetry as $f$, which is different from the case for equivariant CMC surfaces considered by Burstall and Kilian in Section 7 of \cite{Bu-Ki}.
    \end{enumerate}
\end{remark}

\begin{proof}
By Proposition \ref{prop-non-real}, we need only to show that $b_+$ commutes with $D(\lambda)$, or equivalently,  $\chi(\lambda)$ commutes with $D(\lambda)$.
By our assumptions, the translations by arbitrary real $t$ and by integer multiples of $\omega$ induce symmetries of the TE harmonic map $f_{\lambda}$. That is, we have
\[f(z+t,\lambda)=\exp(tD(\lambda))f(z,\lambda) ~\hbox{ and }\ f(z+\omega,\lambda)=\chi(\lambda)f(z,\lambda).\]
Therefore,  the real matrices $\exp(tD(\lambda)) $ and $\chi(\lambda)$ are contained in the image of the monodromy representation of the symmetry group of $f$.
Since all translations in $\C$ commute, we obtain
\[\chi(\lambda)\exp(tD(\lambda))f(z,\lambda)=f(z+t+\omega,\lambda)=\exp(tD(\lambda))\chi(\lambda)f(z,\lambda).\]
Since $f$ is full, we see that $\chi(\lambda)\exp(tD(\lambda))=\exp(tD(\lambda))\chi(\lambda)$.
Differentiating this relation for $t$ at $t=0$, yields $[ D(\lambda), \chi (\lambda)] =0.$
\end{proof}

\begin{remark}
There are clearly many questions concerning TE primitive harmonic maps (and their associated surfaces) with non-real translational  symmetry or TE tori. We plan to pursue these issues in a subsequent publication.
\end{remark}

%%%%%%%%%%%%%%%%%%

\subsubsection{An example for Willmore surfaces with non-real translational  symmetry}

{The simplest examples are the homogeneous Willmore tori, among which the Ejiri torus \cite{Ejiri1982} is of  particular interest:}
 \[y= \left(\cos  u \cos \frac{v}{\sqrt{3}}, \cos u \sin \frac{v}{\sqrt{3}}, \sin  u \cos \frac{v}{\sqrt{3}}, \sin u \sin \frac{v}{\sqrt{3}}, \sqrt{2}\cos \frac{v}{\sqrt{3}}, \sqrt{2}\sin \frac{v}{\sqrt{3}}\right)^t.\]
  It is straightforward to compute
 \begin{equation}
\alpha_{\lambda}'=\left(
                              \begin{array}{ccccccc}
                                0 & 0 & s_1 & s_2 & 0& 0 & \lambda^{-1}\sqrt{2}\beta_3 \\
                                0 & 0 & s_3 & s_4 & 0& 0 & -\lambda^{-1}\sqrt{2}\beta_3 \\
                               s_1& -s_3 & 0 & 0 & -\lambda^{-1}k_1 & -\lambda^{-1}k_2 & 0\\
                               s_2& -s_4 & 0 & 0 & -\lambda^{-1}ik_1 & -\lambda^{-1}ik_2 & 0\\
                               0& 0 & \lambda^{-1}k_1 & \lambda^{-1}ik_1 &0 & 0 & -a_{13}\\
                               0& 0& \lambda^{-1}k_2 & \lambda^{-1}ik_2 & 0 & 0 & -a_{23}\\
                               \lambda^{-1}\sqrt{2}\beta_3 & \lambda^{-1}\sqrt{2}\beta_3 & 0 & 0 &a_{13}  & a_{23} & 0\\
                              \end{array}
                            \right)\dd z,
\end{equation}
with
\[k_1=\frac{\sqrt{6}}{12},k_2=\frac{-i\sqrt{3}}{6}, s=\frac{1}{6}, \beta_3=\frac{-i\sqrt{2}}{24}, a_{13}=\frac{-i\sqrt{3}}{6}, a_{23}=\frac{\sqrt{6}}{6},\]
and
\[s_1=\frac{7\sqrt{2}}{48},\ s_2=\frac{-11i\sqrt{2}}{48},\ s_3=\frac{17\sqrt{2}}{24}, s_2=\frac{-13i\sqrt{2}}{48}.\]
So a holomorphic potential of $y$ is
\[\xi=D(\lambda)\dd z \hbox{ with }D(\lambda)=\alpha_{\lambda}'(\partial_z)+\overline{\alpha_{\lambda}'(\partial_ z)}.\]
Set
\[\mathfrak{A}_{\lambda}=\alpha_{\lambda}'(\partial_z)+\overline{\alpha_{\lambda}'(\partial_ z)},\ \mathfrak{B}_{\lambda}=i\left(\alpha_{\lambda}'(\partial_z)-\overline{\alpha_{\lambda}'(\partial_ z)}\right).\]
Note that in this case, $[\mathfrak{A}_{\lambda}, \mathfrak{B}_{\lambda}]=0$. So we have
\[F(z,{\lambda})=\exp(u\mathfrak{A}_{\lambda}+v\mathfrak{B}_{\lambda} ).\]
One computes in this case
$\mathfrak{A}_{\lambda=1}$ is similar to $diag\{0,0,0,i,-i,i,-i\}$ and $\mathfrak{B}_{\lambda}|_{\lambda=1}$  is similar to $diag\{0,\frac{i}{\sqrt{3}},\frac{-i}{\sqrt{3}},\frac{i}{\sqrt{3}},\frac{-i}{\sqrt{3}},\frac{i}{\sqrt{3}},\frac{-i}{\sqrt{3}}\}$, which means that the resulting surface reduces to a torus, as was shown explicitly in  \cite{Ejiri1982}.

In view of Theorem \ref{secondperiod}, since
\[F(z,{\lambda})=\exp(u\mathfrak{A}_{\lambda}+v\mathfrak{B}_{\lambda} )=\exp(zD(\lambda))\exp(-2iv\overline{\alpha_{\lambda}'(\partial_ z)}),\]
we have
\[\chi(\lambda)=\exp(2\sqrt{3}\pi\mathfrak{B}_{\lambda} ), \hbox{ and } b_+(\lambda)=\exp(-4i\sqrt{3}\pi\overline{\alpha_{\lambda}'(\partial_ z)}),\]which illustrates how Theorem \ref{secondperiod} works.
%%%%%%%%%%%%%%%%%%

%%%%%%%%%%%%%%%%%%%%%%%%%%%%%%%%%%%%%%%%%%%

\section{On equivariant Willmore surfaces from Moebius strips  containing the center line into spheres}

In this section we consider equivariant Willmore Moebius strips    in $S^{n+2}$ and provide  an equivalent  description of all equivariant Willmore Moebius strips. As an illustration of our method we discuss a minimal Willmore Moebius strip in $S^3$ which was  introduced by Lawson \cite{Lawson}. We also prove the existence of many Willmore Moebius strips not conformally equivalent to the Lawson examples.

%MMMMMMMMMMMMMMMMMMMMMMMMMMMMMMMMM
\subsection{Equivariant Willmore surfaces of Moebius strips}

%%%%%%%%%%%%%%%%%%%%%%%%%%%%%%%%%

\subsubsection{Equivariant Willmore surfaces  from a  Moebius strip to $S^{n+2}$ }
We first construct a Willmore surface
$\check{y}: \check{M} \rightarrow S^{n+2}$ defined on  the orientable  double cover
$\check{M}=\St$ of a Moebius strip $M$. Let
$y:\St \rightarrow S^{n+2}$ be a translationally equivariant Willmore surface with its extended frame $F(z,\bar z, \lambda)$ satisfying $F(0,0,\lambda)=I$ and a holomorphic potential
\begin{equation}\label{eq-pot-equ}
\xi = D(\lambda) \dd z, \hspace{2mm} \mbox{where} \hspace{2mm}
D(\lambda) = \lambda^{-1} D_{-1} + D_0 + \lambda D_{1}
\hspace{2mm} \mbox{and where} \hspace{2mm} D_1 = \overline{D_{-1} }
\hspace{2mm} \mbox{and } \hspace{2mm}   D_0 = \overline{D_0} .
\end{equation}
Here $D(\lambda)=\alpha'(0)+\alpha''(0)$ is constant with $\alpha$ being of the form stated in Section 4 (See also \cite{DoWa1}, Proposition 2.2). And the solution \begin{equation}
C(z,\lambda) = \exp{z D(\lambda)}
\end{equation}
to the ODE $\dd C = C \xi,\ C(0,\lambda) = I,$ for all $ \lambda \in S^1$, gives a holomorphic frame of $y$. Moreover, there exists some $t_0\in\R^+$ such that  $y(z+t_0)=y(z)$ for all $z\in \C$ (i.e. $y$ is a Willmore cylinder) if and only if $\exp(t_0D(\lambda))|_{\lambda=1}=I.$

We consider  the transformation
\begin{equation}
\mu(z) = \bar{z} + \pi.
\end{equation}
and restrict our attention to strips $\mathbb{S}$ containing the real axis which are invariant under $\mu$.
The group $ \Gamma = \Gamma_\M = \pi_1 (M)$ generated by $\mu$ satisfies
\begin{equation}\label{eq-gamma-1}
 \Gamma = \Gamma_0 \cup \mu \Gamma_0 \hspace{2mm} \mbox{where }
\hspace{2mm}  \Gamma_0  = \Gamma_{\M,0}= \lbrace  z \rightarrow {z + 2n\pi} \rbrace, \hspace{2mm} n \in \mathbb{Z}.
\end{equation}
Since no element of $\Gamma$ has any fixed point, except the identity transformation,
$\check{M} = \Gamma_0 /\C$ is a smooth real surface, actually
 a  cylinder,  the orientable double cover of the Moebius strip  $M  = \Gamma /\St $  to be constructed.

Let's assume now that $\mu$ induces a symmetry of the translationally equivariant Willmore surface
 $y: \St \rightarrow S^{n+2},$ generated by $D(\lambda)$.
Then from the definition $C(z,\lambda) = \exp(zD(\lambda))$ we obtain
\begin{equation}\label{eq-C-equ}
C(\mu(z),\lambda) = \exp ( \pi D(\lambda)) C(\bar{z},\lambda).
\end{equation}
Since $\mu \circ \mu (z) = z+2 \pi$,
we have moreover
\begin{equation}
C(\gamma_n(z),\lambda) = \exp (2n\pi D(\lambda)) C(z,\lambda),
\end{equation}
if $\gamma_n (z) = z + 2n\pi \in \Gamma_0.$

\begin{theorem}\label{thm-equ-n-ori}\
\begin{enumerate}
\item Let $y:\St \rightarrow S^{n+2}$ be a translationally equivariant  Willmore surface with a symmetry $(\mu,S)$ i.e. satisfying
$ Y\circ \mu= \hat S Y$ for some $\hat S\in SO^+(1,n+3)$. Here $Y$ is the canonical lift of $y$ w.r.t. $z$.
 Let $\xi=D(\lambda)\dd z$ be  the Delaunay  potential \eqref{eq-pot-equ} associated with $y$ as stated above.  Then
    \begin{enumerate}
\item The frame $C(z,\lambda)$ satisfies
\begin{equation} \label{3}
C(\mu(z),\lambda))= \chi (\lambda) \cdot \overline{C(z,\lambda^{-1})} \cdot W_+( \bar{z},\lambda)~~\hbox{ for all $z \in \C$}.
\end{equation}Here
\begin{equation} \label{4}
\begin{split}
W_+ (\bar{z}, \lambda) &= W_+( \bar{z}, \lambda = 1)\\
& = W_0(\bar{z})\in SO^+(1,n+3,\C) \cap (O^+(1,3,\C)\times O(n,\C)),
\end{split}
\end{equation}with  $W_0(u) \in SO^+(1,n+3) \cap (O^+(1,3)\times O(n))$  for $u \in \R.$ Moreover,
the monodromy $\chi(\lambda)$ of $\mu$ satisfies
\begin{equation} \label{5}
\chi (\lambda) =
\exp (\pi D(\lambda)) \cdot  W_0 (0)^{-1}, ~ \chi(\lambda)|_{\lambda=1}=\hat S.
\end{equation}
\item The potential $D(\lambda)$ satisfies
\begin{equation} \label{eq-D-equ}
\dd W_0(\bar{z})=W_0(\bar{z})D(\lambda) -D(\lambda^{-1})  W_0(\bar{z}),
 \hspace{2mm} \mbox{and } \hspace{2mm}
 W_0(0) D(\lambda) = D(\lambda) W_0(0),
\end{equation}
from which we obtain
\begin{equation} \label{eq-D-equ1}
W_0(\bar{z})=\exp(-\bar{z}D(\lambda^{-1}))W_0(0)\exp(\bar{z}D(\lambda))~~ \hbox{ with }~~~ W_0(0)D_{-1}=D_1W_0(0).
\end{equation}
 \end{enumerate}

\item Conversely, if we have some potential $\xi = D(\lambda) \dd z$
 in the form \cite{DoWa1}, Proposition 2.2. which satisfies \eqref{eq-D-equ1} for some $W_0(\bar{z})$, where $W_0(u) \in SO^+(1,n+3) \cap (O^+(1,3)\times O(n))$ for all $u \in \R$, then the Willmore surface $y_\lambda$ defined on some open strip $\St$ containing the real axis
 and associated with $\xi$
admits the transformation $(\mu, \chi (\lambda)) $ as a symmetry, where $\chi(\lambda)=\exp (\pi D(\lambda)) \cdot  W_0 (0)^{-1}$.
\end{enumerate}
\end{theorem}

 \begin{proof}
(1). First applying Theorem 4.4 of \cite{DoWaSym2} to the surface $\check{y}$ with symmetry $(\mu,I)$, we obtain that $C$ satisfies \eqref{3}. Comparing to \eqref{eq-C-equ} we derive
\[\exp ( \pi D(\lambda)) C(\bar{z},\lambda)=\chi(\lambda)\overline{C(z,\lambda^{-1})} W_+( \bar{z},\lambda)
\]
for some $\chi(\lambda)$ and $W_+(\bar{z},\lambda)$. Setting $z=u+i0$, we have
\begin{equation}\label{eq-w+u}
W_+(u,\lambda)=\overline{C(u,\lambda^{-1})}^{-1}\chi(\lambda)^{-1}\exp ( \pi D(\lambda)) C(u,\lambda).
\end{equation}
Since all terms on the right side of this equation are real, we infer $W_+(u,\lambda)=W_+(u)\in SO^+(1,n+3,\C)\cap (O(1,3)\times O(n))$. As a consequence, we obtain that
$W_+(\bar{z},\lambda)=W_0(\bar{z})\in SO^+(1,n+3,\C)\cap (O(1,3,\C)\times O(n,\C))$. Setting, furthermore, $u=0$ in \eqref{eq-w+u} we obtain \eqref{5}. Differentiating both sides of   \eqref{3} yields  \eqref{eq-D-equ}.
It is easy to verify that
$W_0(\bar{z})$ in \eqref{eq-D-equ1} also solves \eqref{eq-D-equ} and satisfies the same initial condition at $z=0$. Comparing the $\lambda$--terms in \eqref{eq-D-equ}, we obtain that $ W_0(0)D_{-1}=D_1W_0(0)$.
This finishes the proof of (1).

(2). The equation \eqref{eq-D-equ1} is equivalent to
\[\mu^*\xi(z,\lambda)=W_0(\bar{z})^{-1}\overline{\xi(z,\lambda^{-1})}W_0(\bar{z})+W_0(\bar{z})^{-1}\dd W_0(\bar{z}).\]
  Integrating this equation with initial condition $C(0,\lambda)=I$, we see that there exists some $\chi(\lambda)$ such that
\[C(\mu(z),\lambda))= \chi (\lambda) \cdot \overline{C(z,\lambda^{-1})} \cdot W_0( \bar{z})
\]
holds for all $z \in \C$. Setting $z=0$, we obtain $\exp(\pi D(\lambda))=\chi(\lambda)W_0(0)$. Since $\exp(\pi D( \lambda))\in \Lambda SO(1,n+3)_{\sigma}$, we see that $\chi(\lambda)=\exp(\pi D(\lambda))W_0(0)^{-1}\in \Lambda SO(1,n+3)_{\sigma}$.
\end{proof}

Recall that by definition \ref{full}  a surface $y$ is full if $Fix(f)\subset SO^+(1,n+3)$ is
{in the center of $SO^+(1,n+3)  = \{I\}$, where $f$ is the conformal Gauss map of $y$. Since a Willmore surface $y$ and its oriented conformal Gauss map $f$ are in a $1-1$ correspondence \cite{Ma}, we see that  $Fix(y)\subset SO^+(1,n+3)$ is $\{I\}$.}
\begin{theorem}\label{thm-equ-n-ori-2}\
\begin{enumerate}
\item We retain the assumptions and results of Theorem \ref{thm-equ-n-ori} $(1)$ and assume moreover that
$y$ descends to a Willmore surface of the cylinder $\Gamma_0/\St$ and also to a Willmore immersion from the Moebius strip
 $\Gamma/\St$ (See \eqref{eq-gamma-1} for the definitions of $\Gamma_0$ and $\Gamma$). Then, up to some conformal transformation, we have
\begin{equation}\label{eq-w0}
W_+ (\bar{z}, \lambda) =W_0(\bar{z})= W_0 (0)=\exp( \pi D( \lambda))|_{\lambda=1}=diag(1,1,1,-1,-1,\pm1,\cdots,\pm1)
\end{equation}
where $W_0 (0) \in SO^+(1,n+3)$.
\item  Conversely, consider some potential $\xi = D(\lambda) \dd z$
which induces a Willmore surface $y$ defined on the cylinder  $\Gamma_0/\St$, where $\St$ is some strip. If  $\St=\mu(\St)$
 and if \eqref{eq-D-equ1} and \eqref{eq-w0} are satisfied for $W_0(\bar{z})=W_0 (0) =diag(1,1,1,-1,-1,\pm1,\cdots,\pm1)\in SO^+(1,n+3)$,  then the transformation $\mu$ leaves the Willmore surface $y= y_\lambda|_{\lambda=1}$ invariant and  thus $y$ descends to a Willmore Moebius strip.
\end{enumerate}
\end{theorem}

\begin{proof}(1).
First we have for the frame of $y$ the relation (See for example Theorem 3.1. of \cite{DoWaSym2})
\begin{equation} \label{trafoF}
{F(\mu(z), \overline{\mu (z) } , \lambda) = \chi(\lambda) F(z, \bar(z), \lambda^{-1}) k(z,\bar{z}).}
\end{equation}
So we obtain $F\circ\mu|_{\lambda=1}={\chi(\lambda) F|_{\lambda=1} k}
$. As a consequence, we have $y\circ\mu|_{\lambda=1}=\chi(\lambda)y |_{\lambda=1}$. Since $y\circ\mu|_{\lambda=1}=y |_{\lambda=1}$, we have $\chi(\lambda)y |_{\lambda=1}=y |_{\lambda=1}$.
Since $y$ is full, we have  $\chi(\lambda)|_{\lambda=1}=I$.
As a consequence,  we have  $W_0(0)=\exp \pi D(\lambda)|_{\lambda=1}$. So \[W_0(0)D(\lambda)|_{\lambda=1}=D(\lambda)|_{\lambda=1}W_0(0).\]
Since $ W_0(0)D_{-1}=D_1W_0(0)$, we see $ W_0(0)D_{0}=D_0W_0(0)$. As a consequence, \[W_+(\bar{z})=\exp(-\bar{z}D(\lambda^{-1}))W_0(0)\exp(\bar{z}D(\lambda))=W_0(0).\]
Then we obtain the first equality in \eqref{eq-w0}.

For the second one, recall that in this case we have $\mu(z)=\bar z+\pi$,  which shows  that the $k$ in (3.1) of Theorem 3.1 \cite{DoWaSym2} is (by taking derivatives for $\mu$) of the form $k=diag(k_1, k_2)$ with $k_1=diag(1,1,1,-1)$.
Substituting $z=0$  and $\lambda = 1$  into (\ref{trafoF})  we obtain
\[\exp\pi D_{\lambda}|_{\lambda=1}=k(0).\]
Since $\det k_2 = -1$, $k_2(0)^2=I_n$ and $k_2(0)\in O(n)$, we have that $k_2$ has only eigenvalues $\pm1$ and the number of $-1$ is odd. So $k_2(0)$ is conjugate to $\hbox{diag}(-1,\cdots,-1,1\cdots,1)$ with odd number of $-1$. Changing $D_{\lambda}$ if necessary, we can assume that $k_2(0)=\hbox{diag}(-1,\pm1,\cdots,\pm1)$ with $\det k_2(0)=-1$.

(2). Using Case (2) of Theorem \ref{thm-equ-n-ori},  we obtain that
 $\chi(\lambda)|_{\lambda=1}=\exp \pi D_{\lambda}|_{\lambda=1}W_0(0)^{-1}=I$.
\end{proof}

Theorem \ref{thm-equ-n-ori-2} gives a characterization of all equivariant Willmore Moebius surfaces in $S^{n+2}.$
In particular, applying this throem to the holomorphic potentials given in Proposition \ref{cor-TEW-S3}, we can classify  all TE Willmore Moebius strips in $S^3$ as below. Note that this generalizes the examples given by Lawson \cite{Lawson}, as we will prove in Section 6.2.

\begin{theorem}\label{thm-moebius-s3}
Let $y$ be a TE Willmore strip in $S^3$ with a holomorphic potential $\xi=D(\lambda)\dd z$ {of the form stated in } Proposition \ref{cor-TEW-S3}.
Assume, moreover, that $y$ descends to a Willmore surface of the cylinder $\Gamma_0/\St$ for some strip
$\St$.  Then the following statements hold.
\begin{enumerate}
\item
$y$ descends to a Willmore immersion from the Moebius   strip  $\mathbb{M} = \Gamma/ \mathbb S,$    where $\Gamma=\Gamma_0\cup\mu\Gamma_0,$  and $\Gamma$ and $\Gamma_0$  are as defined in \eqref{eq-gamma-1} restricted to $\mathbb S$, if and only if there exist positive integers $m,l\in \mathbb{Z}^+$ with $(m,l)=1$ and $2|m$ such that
\begin{equation}\label{eq-moebius}
s_1=\frac{m^2-l^2}{2},~ s_2=0,~  k_1=-\frac{il}{2} \hbox{ and } \beta_1=-\bar{\beta}_1=\frac{\sqrt{2}}{2}i\beta^*\in\R.
\end{equation}
That is,  the holomorphic potential $D(\lambda)$ of $y$ is of the form
\begin{equation}\label{eq-moebius2}
D(\lambda)=\left(
                                  \begin{array}{ccccc}
                                    0 & 0 & - \frac{m^2-2}{2\sqrt{2}} &0&  \beta_{\lambda} \\
                                    0 & 0 & \frac{m^2+2}{2\sqrt{2}}& 0& -\beta_{\lambda}  \\
                                 -  \frac{m^2-2}{2\sqrt{2}} &-\frac{m^2+2}{2\sqrt{2}} & 0 & 0& \frac{ i(\lambda^{-1}l-\lambda l)}{2} \\
                                   0& 0& 0 & 0& - \frac{\lambda^{-1}l+ \lambda l}{2}\\
\beta_{\lambda}  &\beta_{\lambda}   & -\frac{ i(\lambda^{-1}l-\lambda l)}{2} &  \frac{\lambda^{-1} l+ \lambda l}{2}& 0 \\
                                  \end{array}
                                \right).\end{equation}
                         with $\beta_{\lambda}= i(\lambda^{-1}\beta^*-\lambda \beta^*)$.
\item Assume the conditions stated in $(1)$ are satisfied. Then $y$ is conformally equivalent to some minimal surface in some space form if and only if $\beta^* =0$. In this case, $\mathbb S=\C$ and $y$ is conformally equivalent to one of Lawson's minimal Klein bottles in $S^3$.
\item  There exist infinitely many translationally equivariant Willmore Moebius strips
 in $S^3$ which are not conformally equivalent to any minimal surface in any space form.
\end{enumerate}
\end{theorem}

The case (1) of Theorem \ref{thm-moebius-s3} can be derived from the following lemma.

\begin{lemma} \label{lemma-moebius}
We retain the assumptions of Theorem \ref{thm-moebius-s3}. Let $y_\lambda$ be the associated family of $y$ and $D(\lambda)$ its Delaunay matrix.
\begin{enumerate}
 \item                            $D(\lambda)$ and $W_0(\bar{z})=\hbox{diag}(1,1,1,-1,-1)$ solve  \eqref{eq-D-equ1} if and only if
\begin{equation}\label{eq-kb-hatD-3}
\beta_1=-\bar\beta_1, k_1=-\bar{k}_1 \ \& \ s_2=0.
\end{equation}
\item Assume that \eqref{eq-kb-hatD-3} holds. Then
  the corresponding associated family of Willmore surfaces ${y}_{\lambda}$ admits the symmetry $(\mu,\chi(\lambda))$ with $\chi(\lambda)=\exp(\pi D(\lambda))W_0(0)^{-1}$, and the eigenvalues of $D(\lambda)|_{\lambda=1}$ are \[\hbox{$0,\pm 2k_1$ and $\pm\sqrt{-4|k_1|^2-2s_1}$.}\]
In particular, $\chi(\lambda)|_{\lambda=1}=I$ if and only if
$e^{2\pi k_1}=-e^{\pi\sqrt{-4|k_1|^2-2s_1}}=-1,$
i.e.,  if and only if  \eqref{eq-moebius} holds.
\end{enumerate}
\end{lemma}

\begin{proof}(1). By Theorem \ref{thm-equ-n-ori}, we see that the matrices $D(\lambda)$  and $W_0(\bar{z})=W_0(0)$ solve \eqref{eq-D-equ1} if and only if $W_0(0)D(\lambda)-D(\lambda^{-1})W_0(0)=0$, which is equivalent to \eqref{eq-kb-hatD-3} by substituting $D(\lambda)$ into the above equation.

(2). By Theorem \ref{thm-equ-n-ori-2}, $y$ descends to a Willmore Moebius strip, if and only if  (\ref{eq-D-equ1}) and (\ref{eq-w0}), equivalently (\ref{eq-kb-hatD-3}), are satisfied. By (\ref{5}) this is equivalent to
  \[\exp \pi D(\lambda)|_{\lambda=1}=W_0(\bar{z})=diag(1,1,1,-1,-1).\]
  But
$\chi(\lambda)|_{\lambda=1}=\exp(\pi D(\lambda))|_{\lambda=1} W_0(0)^{-1}=I$
if and only if $e^{2\pi k_1}=- e^{2\pi\sqrt{-|k_1|^2-2s_1}}=-1.$

Finally, it is straightforward to verify that, in view of (\ref{eq-kb-hatD-3}) the last equation
is equivalent to the first equation of  (\ref{eq-moebius}).
\end{proof}
Note that this lemma  provides all translationally equivariant Willmore surfaces with symmetries $(\mu,\chi(\lambda))$ and $\chi(\lambda)=\exp(\pi L(\lambda))W_0(0)^{-1}$.

\ \\
\emph{Proof of Theorem \ref{thm-moebius-s3}:}  (1). {See Lemma \ref{lemma-moebius}.}

(2) First we note that  there exists no non-orientable translationally equivariant minimal surfaces in $\R^3$, since the Catenoid is the only translationally  equivariant minimal surface in $\R^3$. Secondly,
by Lemma 1.2 of \cite{Wang-2} (see also \cite{Helein}), $y$ is conformally equivalent to a minimal surface in $S^3$ or $\mathbb{H}^3$ if and only if so is its associated family. As a consequence, the monodromy matrix $\exp( \pi D(\lambda))$ takes value in some subgroup of $SO^+(1,4)$ conjugate to $SO(4)$ for all $\lambda$ if  $y$ is conformally equivalent to a minimal surface in $S^3$ . And the monodromy matrix $\exp( \pi D(\lambda))$ takes value in some subgroup of $SO^+(1,4)$  conjugate to $SO^+(1,3)$   for all $\lambda$ if  $y$ is conformally equivalent to a minimal surface in $\mathbb H^3$. An elementary computation shows that $D(\lambda)$ has the eigenvalues
\begin{equation}
0,\pm\frac{\sqrt{2}}{2}\sqrt{-(m^2+l^2)\pm\sqrt{(m^4+l^4+|\beta_1|^2)-(\lambda^2+\lambda^{-2})(m^2l^2+ \frac{1}{2}|\beta_1|^2)}}.
\end{equation}
 If the complex purely imaginary number $\beta_1$ is not $0$,   then $D(\lambda)$ has four purely imaginary eigenvalues, when $\lambda=1$, and  has two real eigenvalues and two purely imaginary eigenvalues, when $\lambda=i$. But, by what was said above, the type of eigenvalue can not change within the associated family in such a way. This contradiction implies $\beta_1 = 0$.

 We will show in the next subsection, by a concrete computation, that the case
$\beta_1=0$ provides exactly the potentials of Lawson's examples, which were in fact  a motivation for the above results.

(3) By (2), all surfaces obtained with $\beta_1 \neq 0$ lead to non-minimal surfaces.
We thus only need to make sure that $D(\lambda)$ induces a Willmore cylinder defined
on $\C/\Gamma_0$. To achieve this it suffices to make sure that for $\lambda = 1$ the four purely imaginary eigenvalues are in $2 i \pi \mathbb{Z}$. \hfill$\Box$

%\begin{figure}[ht]
%\centering
%$
%\begin{array}{ccc}
%\includegraphics[height=38mm]{b=1-%1.jpg} \,\, & \,\,
%\includegraphics[height=38mm]{b=1-%2.jpg} \,\, & \,\,
%\includegraphics[height=38mm]{b=1-%3.jpg}
%\end{array}
%$
%\caption{Willmore Moebius strip (with %$m=2,l=1$) computed with a numerical %implementation of
%DPW \cite{Brander,B-W} (with %parameter X = bjorlingc(1,  0,  0,  %-1/2,   -2-1/2) in \cite{Brander}).
%Left: The case $(u,v)\in[-%\pi,\pi]\times[-\pi,\pi]$.
%Middle: The case $(u,v)\in[-%\pi,\pi]\times[-\pi/10,\pi/10]$. %Right: The case $(u,v)\in[-%\pi/20,\pi/20]\times[-\pi,\pi]$.
%}
%\label{figureexamples}
%\end{figure}

\begin{figure}[ht]
\centering
$
\begin{array}{ccc}
\includegraphics[height=39mm]{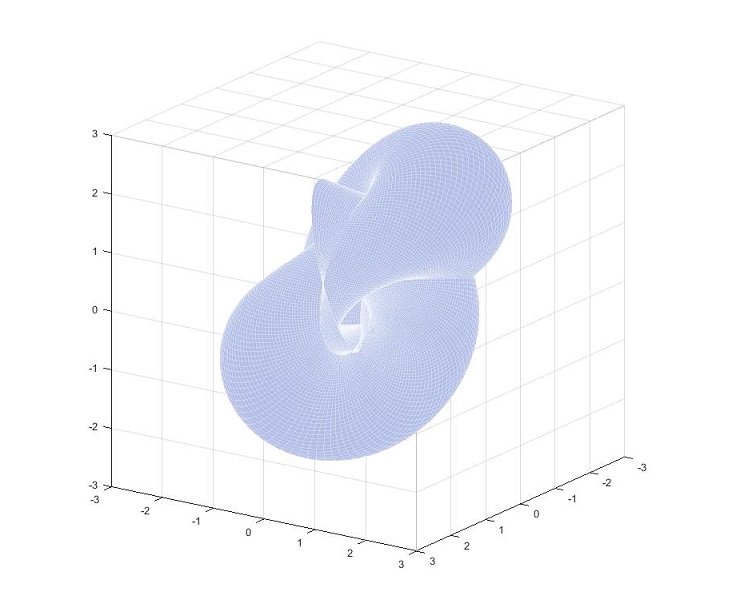} \, & \,
\includegraphics[height=39mm]{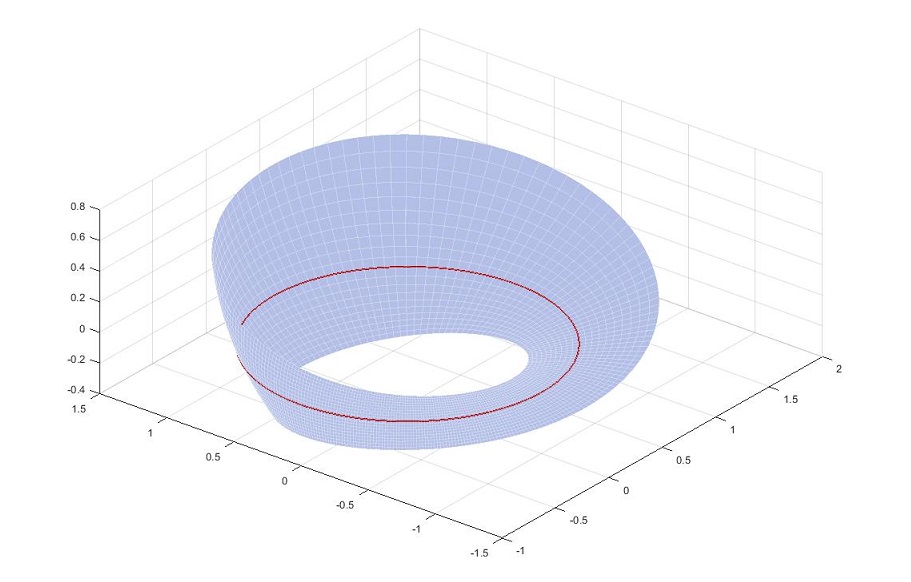} \, & \,
\includegraphics[height=39mm]{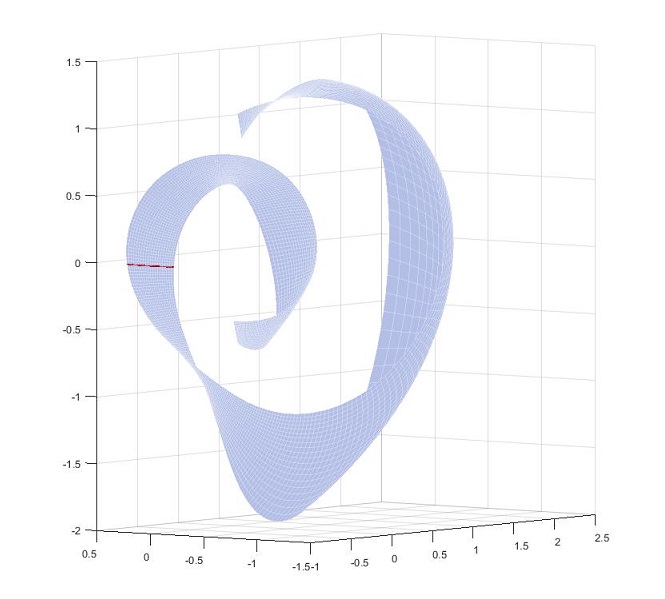}
\end{array}
$
\caption{Willmore Moebius strip (with $m=2,l=1$) computed with a numerical implementation of
DPW \cite{Brander,B-W} (with parameter X = bjorlingc(1,  0,  0,  -1/2,   -2-1/2) in \cite{Brander}).
Left: The case $(u,v)\in[-\pi,\pi]\times[-\pi,\pi]$.
Middle: The case $(u,v)\in[-\pi,\pi]\times[-\pi/10,\pi/10]$. Right: The case $(u,v)\in[-\pi/20,\pi/20]\times[-\pi,\pi]$.
}
\label{figureexamples}
\end{figure}

  We end this subsection with another kind of examples.
\begin{proposition}
Consider the following Delaunay matrices
\begin{equation}\label{eq-kb-hatD-5}D(\lambda)=\left(
                                  \begin{array}{cccccc}
                                    0 & 0 & \frac{\sqrt{2}}{4} &0& \beta_{\lambda}&0 \\
                                    0 & 0 &  \frac{5\sqrt{2}}{4} &0&- \beta_{\lambda}& 0 \\
                                \frac{\sqrt{2}}{4}  &-\frac{5\sqrt{2}}{4}& 0 & 0& \frac{ i(\lambda^{-1}-\lambda )}{2} & -\frac{ \lambda^{-1}+\lambda }{2} \\
                                   0& 0& 0 & 0& - \frac{\lambda^{-1}+ \lambda }{2}&  -\frac{i(\lambda^{-1}- \lambda) }{2}\\
                                     \beta_{\lambda} &  \beta_{\lambda}& -\frac{ i(\lambda^{-1}-\lambda )}{2} &  \frac{\lambda^{-1} + \lambda }{2}& 0&0 \\
                                    0 & 0 &  \frac{\lambda^{-1} + \lambda }{2} & \frac{ i(\lambda^{-1}-\lambda )}{2}& 0&0 \\
                                  \end{array}
                                \right).\end{equation}
         with $\beta_{\lambda}=\sqrt{2}(\lambda^{-1}\beta_1+ \lambda\bar\beta_1)$. Set $\beta_1=i\beta_{10}$ with $\beta_{10}\in\R^+$.
         Then in the associated family $y_\lambda$ of translationally equivariant Willmore surfaces the surface $y = y_{\lambda = 1}$ descends to a Willmore surface from the Moebius strip to $S^4$. This Willmore surface can not be realized as a minimal surface in any space form.
\end{proposition}
\begin{proof} It is not difficult to verify that
        the eigenvalues of $D(\lambda = 1)$ are $\{0,0,\pm 2i, \pm i\}$.
        Therefore the surface descends to the cylinder $\C/ \Gamma_0$.
        Moreover, we have
        \begin{equation} \label{equ}
        \exp(\pi D(\lambda = 1))=W_0(0)=\hbox{diag}(1,1,1,-1,-1,1),
        \end{equation}
         and $\exp(-\bar{z} D(\lambda^{-1}))W_0(0)\exp(\bar{z} D(\lambda))=W_0(0)$.
 By Theorem \ref{thm-equ-n-ori}, the corresponding Willmore surface $y$
 admits the symmetry $(\mu, \chi(\lambda))$. Now (\ref{equ}) implies that $y = y_1$
 is an equivariant Willmore Moebius strip in $S^4$. Moreover, since the rank of $B_1$ is $2$, it is not S-Willmore (see for example \cite{DoWa1}) and hence can not be conformally equivalent to any minimal surface in any space form.
\end{proof}

%%%%%%%%%%%%%%%%%%%%%%%%%%%%%

\subsection{Willmore Moebius strips in $S^3$ associated with Lawson's Klein bottles}

In formula (7.1) of \cite{Lawson}  families of minimal tori and minimal Klein bottles in $S^3$ were given by concrete expressions (also see Theorem 3 of \cite{Lawson} for more information).
Here we use one of the simplest examples of \cite{Lawson} to construct a Willmore Moebius strip in $S^3$, {and to illustrate the theory presented in the last subsection. This also finishes the proof of part $(2)$ of} Theorem \ref{thm-moebius-s3}.
%%%%%%%%%%%%%%%%%%%
\subsubsection{Definition and basic properties of Lawson's minimal Klein bottles into $S^3$}

We consider the minimal immersion $y: \R\times\R \rightarrow S^3 $  given by \cite{Lawson}:
\begin{equation} \label{Klein}
y(u,v) = (\cos(mu)\cos(\hat{v}(v)),\sin(mu) \cos(\hat{v}(v)), \cos(lu) \sin(\hat{v}(v)), \sin(lu) \sin(\hat{v}(v))).
\end{equation}
Here  $m,l\in\mathbb{Z}^+$, $2|ml$ and $(m,l)=1$. If $l$ is an even number, we can change the coordinates such that  the new surface is congruent to the above one after applying an isometry of $S^3$. So w.l.g. we assume that $m>0$ is an even integer.

Note that  $u+i\hat{v}$ is not a complex coordinate of $y$, but if we set
\begin{equation*}
e^{\omega}(\hat{v})=\sqrt{m^2\cos^2\hat{v} +l^2\sin^2\hat{v} }, \hspace{2mm} \ \
v=\int_0^{v} e^{-\omega}(\check{v})\dd \check{v}, \hspace{3mm}   \nu =
2 \int_0^{\pi}e^{-\omega}(\check{v})\dd \check{v},
\end{equation*}
 and let $\hat{v}(v)$ be the inverse function of $v(\hat v)$, then
$z=u+iv$ is a complex coordinate of $y$ with
 $\hat{v}(-v)=-\hat{v}(v),\ \hat{v}(0)=0,\  \hat{v}(v+ {\nu})=\hat{v}(v)+2\pi,$
  and $\frac{\dd\hat{v}}{\dd v}=e^{\omega}.$

Moreover, $y$ has several symmetries:
\begin{enumerate}
\item The immersion $y$ is equivariant relative to the variable $u$ (\cite{Bu-Ki}):
\begin{equation} \label{symmetries}
y(u+t,v) = R(t) y(u,v)~~\hbox{ for all }\ t \in \R.
\end{equation}
Here $R(t)$ rotates continuously the first two coordinates by an angle of $2t$ and the last two coordinates by the angle $t$. Moreover, $R(0)=R(2\pi)=I$  and $R(t)$ is a one-parameter group.

By abuse of notation we denote by $T_1$ and $T_2$ the translations defined below.

\item The immersion $y$ is invariant under the transformations $T_1$ and $T_2$:
\begin{equation}
T_1: (u, v) \mapsto (u +2 \pi, v) \hspace{2mm} \mbox{and} \hspace{2mm}
T_2: (u,v) \mapsto (u, v + \nu).
\end{equation}
\item The immersion $y$ is invariant under the transformation $\mu$:
\begin{equation}
\mu(u,v)=(u + \pi, -v).
\end{equation}

\end{enumerate}

Using these transformations one can define several natural quotients of $\R\times \R$:
\begin{enumerate}
\item
The cylinders
$M_1 = \C/  \mathbb{Z} T_1 ,$ and $M_2 = \C/ \mathbb{Z} T_2;$
\item The torus
$\mathbb{T} = \C/ \tilde{\Gamma}, \hbox{ where }\tilde\Gamma =\{mT_1+nT_2|m,n\in\mathbb Z\};$
\item The  Moebius strip
$\mathbb{M} =  \C / \Gamma,\ \hbox{ where }\Gamma = \Gamma_0 \cup \mu \Gamma_0 \hbox { and }\Gamma_0 = \mathbb{Z} T_1;$

\item  The Klein bottle
$ \mathbb{K} = \C / \Xi,\ \hbox{ where }
\Xi = \tilde{\Gamma} \cup \mu \tilde{\Gamma}.$
\end{enumerate}

To illustrate what was said in Section 6.1, we consider the case of the Moebius strip $\mathbb M$.
%%%%%%%%%%%%%%%%%
\subsubsection{Determining the Delaunay type matrix for $y$ considered as an equivariant surface}
Since $\mathbb{M}$ is derived as quotient of the  equivariant cylinder
$M_1$, we consider this equivariant surface first. Hence we {choose}  a potential of type \eqref{eq-pot-equ}.

Let $F$ denote an extended frame of type (\ref{eq-F}) with Maurer-Cartan form $\alpha$ for $y$. Then we put
$D(\lambda) = \alpha' (0) + \alpha'' (0).$ By the results of \cite{Bu-Ki} this is the Delaunay matrix associated with the surface $F(0,0,\lambda = 1)^{-1}y$ with frame  $F(0,0,\lambda = 1)^{-1}y$.  To determine the matrix coefficients of $D(\lambda)$, we compute  $|y_u|^2=|y_v|^2= m^2\cos^2\hat{v}+l^2 \sin^2\hat{v}=e^{2\omega}.$
Setting
\[n=e^{-\omega}( l\sin \hat{v}\sin  mu, -l\sin\hat{ v}\cos mu, -m\cos \hat{v}\sin  lu, m\cos \hat{v}\cos l u),\]
we derive that  $\langle n, n\rangle=1,\ \langle n, y\rangle=\langle n, y_u\rangle=\langle n, y_v\rangle=0,\ \det(y,y_u,y_v,n)>0.$
Next we compute
$\langle y_{uu}, n\rangle=\langle y_{vv}, n\rangle=0,\ \langle y_{uv}, n\rangle= ml.$
Altogether we obtain
\begin{equation*}\label{eq-moving0}
\left\{\begin {array}{lllll}
y_{zz}&=\ 2\omega_z y_{z}+\Omega n,\\
y_{z\bar{z}}&=\ -\frac{1}{2}e^{2\omega}y,\\
n_z&=\ -2e^{-2\omega}\Omega y_{\bar{z}},\\
\end {array}\right.\ \ \
\end{equation*}
$\hbox{with } \Omega=-\frac{i}{2}\langle y_{uv}, n\rangle= -iml.$ Setting $Y=e^{-\omega}(1,y),$ we have $|Y_z|^2=\frac{1}{2}$ and (see \cite{BPP,DoWa10,DoWa12} for more details)
\begin{equation*}\label{eq-moving0}
\left\{\begin {array}{lllll}
Y_{zz}&=\ -\frac{s}{2}Y+k \psi,\\
Y_{z\bar{z}}&=\ -|k|^2Y+\frac{1}{2}N,\\
N_z&=\ -2|k|^2Y_z-s y_{\bar{z}}+2k_{\bar{z}}\psi,\\
\psi_z &=\  2k_{\bar{z}}Y- 2kY_{\bar{z}},\\
\end {array}\right.
\end{equation*}
with $s=2(\omega_{zz}-\omega_z^2),\  k=e^{-\omega}\Omega=-ie^{-\omega}$ and  $k_{\bar{z}}=ie^{-\omega}\omega_{\bar{z}}.$
Since
$2e^{2\omega}\omega_{ z}=i(m^2-l^2)\cos\hat v\sin\hat v e^{\omega},$
we obtain
\[\omega|_{z=0}=\ln m, \ \omega_z|_{z=0}=0, \ \omega_{zz}|_{z=0}=\frac{m^2-l^2}{4},\ s(0)=\frac{m^2-l^2}{2},\ k(0)=- \frac{il}{2} \hbox{ and }\ k_{\bar{z}}(0)=0.\]
 As a consequence, we obtain for
the frame $F(z,\lambda)$, formed like the frame in  \eqref{eq-F} (see also \cite[Proposition 2.2]{DoWa12}), that
$\alpha(0,\lambda)=\alpha(0,\lambda)'\dd z+\alpha(0,\lambda)''\dd\bar{z}$,
 with $\alpha(0,\lambda)'=\left.F(z,\lambda)^{-1}\partial_z F(z,\lambda)\right|_{z=0}$ being equal to
\[\left(
                                  \begin{array}{ccccc}
                                    0 & 0 & - \frac{m^2-2}{4\sqrt{2}} & \frac{-i(2+m^2-2l^2)}{4\sqrt{2}}& 0 \\
                                    0 & 0 & \frac{m^2+2}{4\sqrt{2}} & \frac{-i(2-m^2+2l^2)}{4\sqrt{2}}& 0 \\
                                  -\frac{m^2-2}{4\sqrt{2}}  & -\frac{m^2+2}{4\sqrt{2}} & 0 & 0& \frac{i\lambda^{-1}l}{2} \\
                                   \frac{-i(2+m^2-2l^2)}{4\sqrt{2}}& \frac{-i(2-m^2+2l^2)}{4\sqrt{2}} & 0 & 0& - \frac{\lambda^{-1}l}{2} \\
                                    0 & 0 & -  \frac{i\lambda^{-1}l}{2} &  \frac{\lambda^{-1}l}{2}& 0 \\
                                  \end{array}
                                \right).\]
Hence
\begin{equation}\label{eq-kb-hatD-0}D(\lambda)=\alpha(0,\lambda)'+\alpha(0,\lambda)''=\left(
                                  \begin{array}{ccccc}
                                    0 & 0 & - \frac{m^2-2}{2\sqrt{2}} &0& 0 \\
                                    0 & 0 & \frac{m^2+2}{2\sqrt{2}}& 0& 0 \\
                                 -  \frac{m^2-2}{2\sqrt{2}} &-\frac{m^2+2}{2\sqrt{2}} & 0 & 0& \frac{ i(\lambda^{-1}l-\lambda l)}{2} \\
                                   0& 0& 0 & 0& - \frac{\lambda^{-1}l+ \lambda l}{2}\\
                                    0 & 0 & -\frac{ i(\lambda^{-1}l-\lambda l)}{2} &  \frac{\lambda^{-1} l+ \lambda l}{2}& 0 \\
                                  \end{array}
                                \right).\end{equation}
So
$\hat\xi=D(\lambda)\dd z$ is the holomorphic potential generating the Willmore surface
$F(0, \lambda = 1)^{-1}y$  conformally equivalent to  $y$, as explained above.

 Since the potential in \eqref{eq-kb-hatD-0} is exactly the case  $\beta_1=0$ in Theorem \ref{thm-moebius-s3}, it  finishes the proof of  Theorem \ref{thm-moebius-s3}.

%%%%%%%%%%%%%%%%%

\begin{remark}\
%\begin{itemize}
%\item
Lawson's examples actually descend to  equivariant Willmore immersions from Klein bottles
to $S^{n+2}.$ Discussing this goes beyond the topic of this paper. We plan to consider this problem in a future publication.
%\end{itemize}
\end{remark}

%%%%%%%%%%%%%%%%%%%%%

{\bf Acknowledgements}\ \
The second named author was partly supported by the Project 12371052 of NSFC. Part of this work was carried out during visits of the first named author to  Fujian Normal University and visits of the second named author to TU Munich. Both authors would like to thank their host Universities for their hospitality. We would like to thank Shimpei Kobayashi for helpful information about helicoidal surfaces.

\

{\bf Data Availability Statement}
The data used to support the findings of this study are included within the article.
\def\refname{References}

\

  Josef F. Dorfmeister

Fakult\" at f\" ur Mathematik, TU-M\" unchen, Boltzmannstr.3, D-85747, Garching, Germany

{\em E-mail address}: 	josef.dorfmeister@tum.de\\

Peng Wang

School of Mathematics and Statistics, FJKLAMA, Key Laboratory of Analytical Mathematics and Applications (Ministry of Education), Fujian Normal University, Fuzhou 350117, P. R. China.

{\em E-mail address}: pengwang@fjnu.edu.cn

\end{document}